\documentclass[11pt,a4paper]{article}
\usepackage[margin=2.5cm]{geometry}
\usepackage{amsmath,amssymb,amsthm,mathtools}
\usepackage{enumitem}
\usepackage{booktabs}
\usepackage{array}
\usepackage{xcolor}
\usepackage{algorithm}
\usepackage{algpseudocode}
\usepackage[colorlinks=true,linkcolor=blue!60!black,citecolor=green!40!black,urlcolor=blue!60!black]{hyperref}
\usepackage{float}
\usepackage[numbers,sort&compress]{natbib}
\usepackage{cleveref}
\allowdisplaybreaks
\usepackage{authblk}

\newtheorem{theorem}{Theorem}[section]
\newtheorem{proposition}{Proposition}[section]
\newtheorem{lemma}{Lemma}[section]
\newtheorem{corollary}{Corollary}[section]

\theoremstyle{definition}
\newtheorem{definition}{Definition}[section]
\newtheorem{example}{Example}[section]
\theoremstyle{remark}
\newtheorem{remark}{Remark}[section]
\newtheorem{assumption}{Assumption}

\numberwithin{equation}{section}

\newcommand{\eps}{\varepsilon}
\newcommand{\E}{\mathbb{E}}
\newcommand{\PP}{\mathbb{P}}
\newcommand{\QQ}{\mathbb{Q}}
\newcommand{\R}{\mathbb{R}}
\newcommand{\dd}{\mathrm{d}}
\newcommand{\1}{\mathbf{1}}
\newcommand{\TV}{\operatorname{TV}}
\newcommand{\Var}{\operatorname{Var}}
\newcommand{\Cov}{\operatorname{Cov}}
\newcommand{\sgn}{\operatorname{sgn}}
\newcommand{\Leb}{\operatorname{Leb}}
\newcommand{\dist}{\operatorname{dist}}
\newcommand{\osc}{\operatorname{osc}}

\newcommand{\bb}{\bar b}
\newcommand{\drift}{\mathfrak{b}}
\newcommand{\vth}{\vartheta_\eps}
\newcommand{\Hol}{\mathcal{C}^{0,\alpha}}
\newcommand{\BVJ}{\mathrm{BV}_{\!\mathrm J}}
\newcommand{\Cusp}{\mathcal{K}^{\alpha}_{S}}
\newcommand{\Phibar}{\overline{\Phi}}
\newcommand{\seg}[2]{[\![#1,#2]\!]}

\title{Control variates for dry friction under colored noise:\\[2pt]
variance rates for non-smooth observables}
\author{Lihong Guo \thanks{Email: lihguo2223@cityu.edu.cn}} 
\affil{\small\it City University of Hong Kong Shenzhen Research Institute, Shenzhen, China}
\date{}

\begin{document}
\maketitle

\begin{abstract}
A control variate method was introduced by Garnier and Mertz (2022) 
for multivalued stochastic differential equations driven by colored noise. 
For smooth observables, the asymptotic variance is $O(\varepsilon^{2})$, 
where $\varepsilon^{2}$ is the correlation time of the noise. 
However, for non-smooth reliability observables, 
such as threshold indicators and overshoot severities, 
rate estimates have not been established in general. 

In this paper, rate estimates are derived for the inertial dry-friction inclusion 
driven by colored noise in the three regimes whose white-noise propagators are known. 
The asymptotic variance of the resulting control variate estimators is 
$O(\varepsilon^{2\alpha})$ for $\alpha$-H\"older observables with $\alpha\in(0,1]$, 
$O(\varepsilon)$ for observables of bounded variation satisfying a one-sided continuity condition, 
$O(\varepsilon^{1+\alpha})$ for observables that are both $\alpha$-H\"older and of bounded variation, 
and $O(\varepsilon^{\min\{2,1+2\alpha\}})$ for observables with finitely many $\alpha$-cusps, 
where $\alpha \in (0, 1]$ is the cusp exponent. 
The last three bounds hold up to logarithmic factors. 
Bounds on the bias of the white-noise approximation and 
non-asymptotic sample-size estimates are also derived, 
and the comparison principle is shown to be preserved under coupled time discretization. 
Numerical experiments in the three regimes support the predicted rates.
\end{abstract}



\section{Introduction}\label{sec:intro}
Dry (Coulomb) friction governs stick--slip oscillations and the positioning accuracy of machines \cite{FGHP98,HE17}. 
While a body slides, the friction force has magnitude $\Delta>0$ and opposes motion. 
At rest, it takes any value in $[-\Delta,\Delta]$ that keeps the body at rest. 
This is described by the subdifferential of $\varphi(v)=\Delta|v|$, 
with $\partial\varphi(v)=\{\Delta\sgn(v)\}$ for $v\neq0$ and $\partial\varphi(0)=[-\Delta,\Delta]$. 
The equation of motion is a differential inclusion \cite{MM93,Ste00}, 
well posed by maximal monotone operator theory \cite{Bre73,PR14}. 
Random loads lead to stochastic differential inclusions \cite{Cep95,Cep98,PR14}.

Under Gaussian white noise, dry friction is well understood. De Gennes \cite{dG05}
studied the stationary velocity law and the velocity correlations. 
The transition density is known in closed form for pure dry friction \cite{CD61,KS84,TVJ10}.
Touchette, Van der Straeten and Just \cite{TVJ10} solved the
Fokker--Planck equation by spectral series when viscous friction and a constant force are
added, and Touchette, Prellberg and Just \cite{TPJ12} obtained the Laplace transform of the propagator in closed form.
See also \cite{BCT10,BTC11,MG11} and the experiments \cite{GMC09,GC10,Gno13}.
Physical forcing, however, is colored \cite{LY87,HJ95}. 
A finite correlation time changes the dynamics qualitatively. 
It induces genuine sticking \cite{JH87,GJ17}, 
leads to a piecewise-deterministic description with a critical correlation time \cite{Dav84,GLM23}, 
and can make the white-noise approximation inaccurate for statistics such as the diffusivity \cite{GM23}. 
Thus white-noise results cannot replace computations at positive correlation time.

The forcing can be represented by a linear functional of a fast multivariate Ornstein--Uhlenbeck (OU) process
\cite{GM22}. 
Let $\eps^{2}$ be its correlation time. As $\eps\downarrow0$, 
the response converges to the solution of a white-noise stochastic differential equation by the
diffusion approximation \cite{PK74,FGPS07,PS08}, 
with the Wong--Zakai interpretation of the limit \cite{WZ65}. 
For fixed $\eps$, expectations are costly. 
The Kolmogorov equation lives in the augmented phase space and has coefficients of order $\eps^{-2}$. 
Plain Monte Carlo is slow for the small probabilities of reliability analysis. 
Control variates \cite{LW81,Nel90,Gla04,GL09} reduce this cost when a correlated quantity with a computable mean is available. 
Garnier and Mertz \cite{GM22} drive the colored-noise system and its white-noise limit by the same Wiener process. 
They compute the expectation of the limit from its Kolmogorov equation, which has no large terms, 
and estimate only the small difference by Monte Carlo. 
For smooth observables and the systems considered here, 
the asymptotic variance is $O(\eps^{2})$ \cite{GM22}. 
Other variance reduction methods require structure that is not natural here. 
Importance sampling \cite{AG07} needs a biasing measure, 
subset simulation \cite{AB01} intermediate events, 
and multilevel Monte Carlo \cite{Keb05,Gil08,Gil15} a hierarchy of resolutions along which the observable varies smoothly.

Reliability analysis, however, needs non-smooth observables \cite{MSW19,LM19,IMMW26}.
Examples are the indicator $\1_{\{|v|\ge v_f\}}$ of threshold crossing and severity measures that
quantify overshoot. The latter are continuous but singular at the threshold. For such
observables, the experiments of \cite{GM22} indicate a variance of order $\eps$ rather than
$\eps^{2}$. No rate has been proved. The difficulty is structural. The perturbed test function
method of \cite{GM22} requires observables with bounded derivatives. For an indicator, the
difference along the coupled trajectories is not small but rarely nonzero. Its second moment
is essentially the probability that the threshold lies between the two terminal velocities.
Bounding this probability requires regularity of the white-noise law near the threshold and
tail bounds for the coupling error. The mean-square coupling bound of \cite{GM22} alone cannot
give a better rate than $\eps^{2/3}$. Indeed, for $f$ of bounded variation, Avikainen's
inequality \cite{Avi09} bounds the $L^{q}$ distance between the coupled values of $f$ by a
moment of the coupling error and a density bound. With the second moment it gives
$\eps^{2/3}$, and it is sharp under this information; see also \cite{GHM09,TTY22}. We obtain
sharper rates for the inertial dry-friction system
\begin{equation}\label{eq:model-intro}
	\ddot X^{\eps}(t)+\partial\varphi\big(\dot X^{\eps}(t)\big)\ni
	b\big(X^{\eps}(t),\dot X^{\eps}(t)\big)+\sigma\,\xi^{\eps}(t),\qquad
	\big(X^{\eps}(0),\dot X^{\eps}(0)\big)=(x_0,v_0),\quad t\in[0,T],
\end{equation}
where $\xi^{\eps}$ is a colored noise of correlation time $\eps^{2}$, $\sigma>0$, and $b$ is
one of the three drifts for which the white-noise problem was solved in \cite{TVJ10}: 
$b=0$ (regime~(I), pure dry friction), 
$b(x,v)=-\gamma v$ with $\gamma>0$ (regime~(II), dry and viscous friction), 
and $b(x,v)=-\gamma v+\ell$ with $\gamma>0$, $\ell\in\R$ (regime~(III), 
dry and viscous friction with a constant force). 
Since $b$ depends only on velocity, 
$V^{\eps}=\dot X^{\eps}$ solves a scalar multivalued equation \cite{GM22}. 
This scalar structure is exploited in our
analysis.

\subsection{Contributions}\label{sec:contrib}
Throughout, $V^{\eps}$ and $V^{0}$ are the colored-noise velocity and its white-noise
counterpart driven by the same Wiener process.  

\paragraph{Main result.}
\Cref{thm:main} bounds the second moment $\E[(f(V^{\eps}(T))-f(V^{0}(T)))^{2}]$, which in turn bounds the
asymptotic variances of the control variate estimators of \cite{GM22}. The bound holds for four
classes of observables. Up to logarithmic factors in the last three cases,
\begin{equation*}\label{eq:rates-intro}
	\E\big[(f(V^{\eps}(T))-f(V^{0}(T)))^{2}\big]\lesssim
	\begin{cases}
		\eps^{2\alpha}, & f\ \alpha\text{-H\"older},\\
		\eps, & f\ \text{of bounded variation},\\
		\eps^{1+\alpha}, & f\ \alpha\text{-H\"older and of bounded variation},\\
		\eps^{\min\{2,1+2\alpha\}}, & f\ \text{with finitely many $\alpha$-cusps},
	\end{cases}
\end{equation*}
where observables of bounded variation satisfy a mild one-sided continuity condition. 
The first rate extends the rate $\eps^{2}$ of \cite{GM22} from smooth to Lipschitz observables ($\alpha=1$). 
The other three go beyond the $\eps^{2/3}$ barrier. 
Up to a logarithm, 
the bound is of order $\eps$ for threshold indicators and of order $\eps^{2}$ for severity measures with $\alpha>1/2$. 
For threshold indicators, this is the order indicated by the experiments of \cite{GM22}.

The proof uses two structural properties. 
First, a pathwise comparison principle shows that
the solution map from integrated forcing to velocity is $2$-Lipschitz for the uniform norm (\Cref{prop:comparison}); 
its continuity is known \cite[Proposition~4.16]{PR14}.
The integrated colored and white forcings differ by an explicit functional of the OU process of order $\eps$. 
Hence the coupling error has Gaussian tails at scale $\vth$ (\Cref{cor:coupling}). 
Second, the law of the white-noise terminal velocity has a density with an explicit bound in all three regimes (\Cref{prop:density}).
It is obtained by comparing the white-noise velocity, 
through Girsanov's theorem, 
with pure dry friction, whose transition density is known in closed form \cite{KS84}. 
In regimes (II) and (III), 
the propagator is known through spectral series \cite{TVJ10} and through its Laplace transform \cite{TPJ12}, 
which give no explicit bound; 
see \Cref{rem:density} for general density estimates.

\paragraph{Consequences.}
The rates imply bounds on the bias of the white-noise approximation, on variance reduction,
and on sample size (Corollaries \ref{cor:bias} to \ref{cor:bernstein}). 

\paragraph{Numerical experiments.}
The coupled time discretization satisfies a discrete comparison principle
(\Cref{prop:discrete}). Thus the coupling survives discretization. Discretization
affects the bias of the estimators but not the coupling error, so strong rates for such
schemes \cite{EKKL22} are not needed. 
In the three regimes, we test the predicted exponents, 
the coupling-error bounds and the variance reduction (Section~\ref{sec:numerics}).

\subsection{Organization and notation}
The paper is organized as follows. 
Section~\ref{sec:model} sets up the model and estimators. 
Section~\ref{sec:main} states the main results. 
Section~\ref{sec:proofs} contains the proofs. 
Section~\ref{sec:numerics} reports numerical experiments. 
Section~\ref{sec:concl} concludes. 
Appendix~\ref{app:aux} contains auxiliary lemmas, 
Appendix~\ref{app:whitenoise} the comparison principle, 
the coupling error and the density bound, 
and Appendix~\ref{app:implementation} implementation details.

We write $a\wedge b=\min\{a,b\}$, $a\vee b=\max\{a,b\}$, $x_+=\max\{x,0\}$, and
$\sgn(x)=x/|x|$ for $x\ne0$, with $\sgn(0)=0$. The set $\seg{x}{y} \triangleq [x\wedge y,x\vee y]$ is
the closed segment between $x,y\in\R$. 
We denote by $|\cdot|$ the Euclidean norm, by $\|\cdot\|$ the operator norm, by $\Leb$ Lebesgue measure, and by
$\Phibar(x) \triangleq (2\pi)^{-1/2}\int_x^{\infty}e^{-z^{2}/2}\,\dd z$ the standard Gaussian tail. 
The space $C^{0}_{0}([0,T])$ consists of the continuous $w:[0,T]\to\R$ with $w(0)=0$, 
with norm $\|w\|_{\infty}=\sup_{t\le T}|w(t)|$.  
Constants $C$ may change from line to line.

\section{Model and mathematical preliminaries}\label{sec:model}
Throughout, $T>0$, $x_0,v_0\in\R$, $\Delta>0$ and $\sigma>0$ are fixed, 
and $\varphi(v)=\Delta|v|$. 

\subsection{The dry-friction inclusion}\label{sec:regimes}
\begin{assumption}[velocity-dependent drift]\label{ass:vel}
	The drift in \eqref{eq:model-intro} is $b(x,v)=\bb(v) \triangleq -\gamma v+\ell$, where either $\gamma=\ell=0$
	(regime~(I)), or $\gamma>0$ and $\ell=0$ (regime~(II)), or $\gamma>0$ and $\ell\in\R$
	(regime~(III)).
\end{assumption}

Under \Cref{ass:vel}, system \eqref{eq:model-intro} is equivalent to
\begin{equation}\label{eq:reduced}
	\dot X^{\eps}=V^{\eps},\qquad
	\dot V^{\eps}+\partial\varphi(V^{\eps})\ni\bb(V^{\eps})+\sigma\xi^{\eps}(t),\qquad
	(X^{\eps}(0),V^{\eps}(0))=(x_0,v_0), \quad t\in [0,T]. 
\end{equation}
The velocity inclusion is autonomous and the position is recovered by quadrature. 
We want to compute $\mathrm{I}^{\eps} \triangleq \E[f(V^{\eps}(T))]$ for an observable $f$ of the terminal
velocity in one of the classes of \Cref{sec:observables}.  
We use the convex potential
\begin{equation}\label{eq:potential}
	U(v) \triangleq \tfrac{\gamma}{2}v^{2}-\ell v+\Delta|v|,\qquad
	U_* \triangleq \inf_{v\in\R}U(v)=
	\begin{cases}0,&\text{regime (I)},\\[2pt]
		-\dfrac{\big((|\ell|-\Delta)_+\big)^{2}}{2\gamma},&\text{regimes (II)--(III)},\end{cases}
\end{equation}
so that $\bb-\partial\varphi=-\partial U$: the noiseless velocity dynamics is the subdifferential flow of $U$. 
In regime~(III), $|\ell|<\Delta$ is the deterministic stick
phase (the minimum of $U$ is attained at $0$) and $|\ell|>\Delta$ the slip phase.

\subsection{Colored noise and coupling}\label{sec:noise}
\begin{assumption}[colored noise]\label{ass:colo}
	Let $d,d'\ge1$, 
	let $\mathbf{A} \in\R^{d\times d}$ have eigenvalues with strictly positive real part, 
	let $\mathbf{K} \in\R^{d\times d'}$, and let 
	$\mathbf{r} \in \R^{d}\setminus\{0\}$. 
	On a filtered probability space $(\Omega,\mathcal F,(\mathcal F_t)_{t\ge0},\PP)$ satisfying the
	usual conditions we are given a standard $d'$-dimensional $(\mathcal F_t)$-Brownian motion
	$\mathbf{W}$ and an $\mathcal F_0$-measurable random vector $\boldsymbol{\eta}_0\sim\mathcal N(0,\mathbf C)$, 
	where $\mathbf C \triangleq  \int_0^{\infty}e^{-\mathbf{A}u}\mathbf{K}$ $\mathbf{K}^{\top}e^{-\mathbf{A}^{\top}u}\,\dd u$ 
	is the unique solution of $\mathbf{A}\mathbf {C}+\mathbf {C}\mathbf{A}^{\top}=\mathbf{K}\mathbf{K}^{\top}$. 
	For $\eps>0$, $\boldsymbol{\eta}^{\eps}$ solves
	\begin{equation}\label{eq:OU}
		\dd\boldsymbol{\eta}^{\eps}(t)
		=-\frac{\mathbf{A}}{\eps^{2}}\,\boldsymbol{\eta}^{\eps}(t)\,\dd t+\frac{\mathbf{K}}{\eps}\,\dd \mathbf{W}(t),
		\qquad \boldsymbol{\eta}^{\eps}(0)=\boldsymbol{\eta}_0, 
	\end{equation}
	and the colored noise is $\xi^{\eps}(t) \triangleq  \mathbf{r}\cdot \boldsymbol{\eta}^{\eps}(t)/\eps $. 
\end{assumption}

The process $\boldsymbol{\eta}^{\eps}$ is stationary,  
$\widehat{\boldsymbol{\eta}}(u) \triangleq \boldsymbol{\eta}^{\eps}(\eps^{2}u)$ is the
stationary OU process 
$\dd\widehat{\boldsymbol{\eta}}=-\mathbf{A}\widehat{\boldsymbol{\eta}}\,\dd u+\mathbf{K}\,\dd\widehat{\mathbf{W}}$ of unit time
scale, driven by the Brownian motion $\widehat{\mathbf{W}}(u) \triangleq \eps^{-1}\mathbf{W}(\eps^{2}u)$. 
Thus $\xi^{\eps}$ has amplitude $O(\eps^{-1})$ and correlation time $O(\eps^{2})$, 
and converges to a white noise; 
$d\ge2$ covers sums of Lorentzian spectra and Kanai--Tajimi filtered noises
\cite{LY87}, \cite[Examples~3.1--3.4]{GM22}. 
Since $\mathbf{A}$ is invertible, 
integrating \eqref{eq:OU} and multiplying on the left by $\eps\mathbf{r}^{\top} \mathbf{A}^{-1}$ gives, 
almost surely for all $t\ge0$,
the identity on which the coupling rests: 
\begin{equation}\label{eq:integrated}
	\int_0^{t}\xi^{\eps}(u)\,\dd u
	= \mathbf{r}^{\top}\mathbf{A}^{-1}\mathbf{K}\,\mathbf{W}(t)-\eps\,\mathbf{r}^{\top}
	\mathbf{A}^{-1}\big(\boldsymbol{\eta}^{\eps}(t)-\boldsymbol{\eta}_0\big).
\end{equation}

We set $\mathbf{c}  \triangleq  \mathbf{A}^{-\top}\mathbf{r} \in\R^{d}\setminus\{0\}$, 
$\kappa_0 \triangleq \sigma|\mathbf{c}|>0$, and define the
integrated forcings and their difference,
\begin{equation*}\label{eq:forcings}
	N^{\eps}(t) \triangleq \sigma\!\int_0^{t}\xi^{\eps}(u)\,\dd u,\quad
	N^{0}(t) \triangleq \sigma\,\mathbf{c}^{\top}\mathbf{KW}(t),\quad
	h^{\eps}(t) \triangleq N^{\eps}(t)-N^{0}(t)=-\eps \sigma\,\mathbf{c}^{\top}\big(\boldsymbol{\eta}^{\eps}(t)-\boldsymbol{\eta}_0\big),
\end{equation*}
so that
\begin{equation}\label{eq:h-bound}
	\sup_{u\le t}|h^{\eps}(u)|\le 2\eps\kappa_0\,\sup_{u\le t}|\boldsymbol{\eta}^{\eps}(u)|,\qquad t\ge0 .
\end{equation}
The process $N^{0}$ is a real Brownian motion with variance $\mathfrak{s}^{2}$ per unit time, where
$\mathfrak{s} \triangleq \sigma|\mathbf{K}^{\top}\mathbf{c}|=\sigma|\mathbf{K}^{\top}\mathbf{A}^{-\top}\mathbf{r}|$.

\begin{assumption}[non-degenerate white-noise limit]\label{ass:nondeg}
	$\mathfrak{s}>0$.
\end{assumption}
Under \Cref{ass:nondeg}, 
$B \triangleq (\mathbf{K}^{\top}\mathbf{c})\cdot \mathbf{W}/|\mathbf{K}^{\top}\mathbf{c}|$ 
is a standard real $(\mathcal F_t)$-Brownian
motion and $N^{0}=\mathfrak{s}B$. 
By \eqref{eq:h-bound} and \Cref{lem:OUmax}, $N^{\eps}\to N^{0}$
uniformly on $[0,T]$ in every $L^{p}$, $p>0$; 
since the forcing is additive, the limit carries no
It\^o--Stratonovich correction.

\subsection{Solution concept and coupled pair}\label{sec:solution}
Since $N^{0}$ is a Brownian path, the velocity inclusion is written in integrated form.

\begin{definition}[Solution]\label{def:solution}
	Let $w\in C^{0}_{0}([0,T])$. A solution of
	\begin{equation}\label{eq:SP}
		\dd v(t)+\partial\varphi(v(t))\,\dd t\ni\bb(v(t))\,\dd t+\dd w(t),\qquad v(0)=v_0,
	\end{equation}
	is a pair $(v,k)$, where $v\in C([0,T])$ and $k:[0,T]\to\R$ is Lebesgue measurable with
	$k(t)\in\partial\varphi(v(t))$ for a.e.\ $t$, such that
	\begin{equation}\label{eq:SP-integral}
		v(t)=v_0+\int_0^{t}\big(\bb(v(u))-k(u)\big)\dd u+w(t),\qquad t\in[0,T].
	\end{equation}
	Necessarily $|k|\le\Delta$ a.e.; $k$ is the friction force.
\end{definition}

If $w\in C^{1}$, then $v$ is absolutely continuous and
$\dot v+\partial\varphi(v)\ni\bb(v)+\dot w$ a.e. 
Since $\varphi$ is finite and Lipschitz, 
\eqref{eq:SP} is a generalized Skorokhod problem, which is well posed by \cite[Theorem~4.17]{PR14}: 
for every $w\in C^{0}_{0}([0,T])$ it has a unique solution $(v,k)$,
$k$ being unique up to Lebesgue-null sets. 
We write $v=\mathcal S(w)$. 
By uniqueness, the solution map $\mathcal S$ is non-anticipative, 
i.e., $\mathcal S(w)|_{[0,t]}$ depends only on $w|_{[0,t]}$. 
The {coupled pair}
\begin{equation}\label{eq:coupled}
	V^{\eps} \triangleq \mathcal S(N^{\eps}),\qquad V^{0} \triangleq \mathcal S(N^{0}),
\end{equation}
with friction forces $k^{\eps}$ and $k^{0}$, consists of continuous adapted processes
(by non-anticipativity and the continuity of $\mathcal S$, Proposition~\ref{prop:comparison}); 
$V^{\eps}$ solves the velocity inclusion in \eqref{eq:reduced} pathwise, 
$V^{0}$ solves $\dd V^{0}+\partial\varphi(V^{0})\,\dd t\ni\bb(V^{0})\,\dd t+\dd N^{0}(t)$, 
and \eqref{eq:coupled} is the coupling of \cite{GM22}. 
We call
$Z^{\eps}(t)\triangleq V^{\eps}(t)-V^{0}(t)$, $t\in[0,T]$, the coupling error. 
Under colored forcing, the particle may stick: if $V^{\eps}\equiv0$ on an interval, 
then $k^{\eps}=\ell+\sigma\xi^{\eps}\in[-\Delta,\Delta]$ on that interval, 
whereas under white forcing it does not (Lemma~\ref{lem:nostick}).

\subsection{Classes of observables}\label{sec:observables}
\begin{definition}[Observables]\label{def:classes}
	Let $f:\R\to\R$.
	\begin{enumerate}[label=(\alph*),itemsep=0pt,topsep=2pt]
		\item \emph{H\"older class.} For $\alpha\in(0,1]$, define 
		$[f]_{\alpha} \triangleq \sup_{x\neq y}|f(x)-f(y)|/|x-y|^{\alpha}$ and
		$\Hol \triangleq \{f:[f]_{\alpha}<\infty\}$. 
		Here \(\mathcal C^{0,\alpha}\) is understood without any boundedness requirement. 
		
		\item \emph{Bounded variation.} Define $\TV(f) \triangleq \sup \sum_{i=1}^{n}|f(x_i)-f(x_{i-1})|$, 
		where the supremum is taken over all \(n\ge1\) and all finite increasing
		sequences \(-\infty<x_0<x_1<\cdots<x_n<\infty\), 
		and set $\mathrm{BV} \triangleq \{f : \TV(f)<\infty\}$.
		
		\item \emph{Condition {\rm (J)}.} For \(f\in\mathrm{BV}\), the one-sided limits 
		$f(x-)=\lim_{y\uparrow x}f(y)$ and $f(x+)=\lim_{y\downarrow x}f(y)$ exist for every \(x\in\R\). 
		We say that \(f\in\mathrm{BV}\) satisfies condition {\rm (J)} if
		$f(x)\in\{f(x-),f(x+)\}$ for every $x\in\R$, 
		and define $\BVJ \triangleq \{f\in\mathrm{BV}: f\text{ satisfies (J)}\}$.
		
		\item \emph{Cusp class.} Fix a finite nonempty set \(S\subset\R\) and \(\alpha\in(0,1]\).
		Let \(\mathrm{AC}_{\mathrm{loc}}(\R)\) denote the set of functions that are absolutely continuous on every compact interval.
		For \(f\in\mathrm{AC}_{\mathrm{loc}}(\R)\), define
		$[f]_{\alpha,S}
		\triangleq
		\operatorname*{ess\,sup}_{x\in\R\setminus S}
		\dist(x,S)^{1-\alpha}|f'(x)|$, 
		where \(f'\) denotes the derivative of \(f\), which exists almost everywhere.
		Then define
		\[
		\Cusp
		\triangleq
		\left\{
		f\in\mathrm{AC}_{\mathrm{loc}}(\R):
		[f]_{\alpha,S}<\infty
		\right\}.
		\] 
	\end{enumerate}
\end{definition}
For \(f\in\mathrm{BV}\), recall that \(f\) is bounded and Borel, 
has at most countably many discontinuities, and admits finite limits at \(\pm\infty\). Write
$f_+(x) \triangleq f(x+)$, $f_-(x) \triangleq f(x-)$, and 
$f(-\infty)\triangleq \lim_{x\to-\infty}f(x)$. 
Then there exists a unique finite signed Borel measure \(\mu_f\) on \(\mathbb{R}\) such that, 
for every \(x\in\mathbb{R}\),
\begin{equation}\label{eq:muf}
	\mu_f\big((-\infty,x]\big)=f_+(x)-f(-\infty),\quad
	\mu_f(\{x\})=f_{+}(x)-f_{-}(x),\quad
	|\mu_f|(\mathbb{R})=\operatorname{TV}(f_+)\le\operatorname{TV}(f),
\end{equation}
see \cite[Section~3.5]{Fol99}. 
Condition {\rm (J)} holds for continuous functions and for right- or left-continuous functions; 
it excludes spurious values at removable discontinuities, such as those of \(\1_{\{0\}}\).

For fixed finite nonempty \(S\subset\R\) and \(\alpha\in(0,1]\), 
functions in \(\Cusp\) are locally Lipschitz on \(\R\setminus S\), 
may have \(\alpha\)-Hölder cusps at the points of \(S\), and belong to \(\mathcal C^{0,\alpha}\) 
(Lemma~\ref{lem:cusp}); for \(\alpha=1\), \(\mathcal K^1_S\) is the class of Lipschitz functions on \(\R\).

\begin{example}[Test observables]\label{ex:observables}
	Fix a threshold \(v_f>0\).
	\begin{enumerate}[label=(\alph*),itemsep=0pt,topsep=2pt]
		\item The \emph{crossing indicator} \(f_{\mathrm{ind}}(v)\triangleq \1_{\{|v|\ge v_f\}}\) belongs to \(\BVJ\), with \(\TV(f_{\mathrm{ind}})=2\), but
		\[
		f_{\mathrm{ind}}\notin\mathcal C^{0,\alpha}
		\qquad\text{for every }\alpha\in(0,1].
		\]
		
		\item The \emph{severity family}
		\[
		f_\alpha(v)\triangleq\min\{((|v|-v_f)_+)^\alpha,1\},
		\qquad \alpha\in(0,1],
		\]
		satisfies \(0\le f_\alpha\le1\), \(\TV(f_\alpha)=2\), and \([f_\alpha]_\alpha\le1\). Moreover,
		\[
		f_\alpha\in\mathcal K^\alpha_{\{\pm v_f\}},
		\qquad
		[f_\alpha]_{\alpha,\{\pm v_f\}}=\alpha.
		\]
		For \(\alpha=1\), \(f_\alpha\) is the capped linear overshoot; 
		smaller \(\alpha\) model damage measures that saturate in the overshoot, 
		as wear and fatigue measures do, and
		\[
		f_\alpha\to f_{\mathrm{ind}}
		\quad\text{pointwise on }\{|v|\neq v_f\}
		\quad\text{as }\alpha\downarrow0.
		\]
		
		\item The \emph{smooth anchor} \(f_{\mathrm{sm}}(v)\triangleq\tanh((v/v_f)^2)\) is Lipschitz, 
		has bounded derivatives of all orders, and satisfies \(\TV(f_{\mathrm{sm}})=2\).
	\end{enumerate}
\end{example}

\subsection{Estimators}\label{sec:estimators}
Let $(\mathbf{W}^{i}, \boldsymbol{\eta}_0^{i})_{i\ge1}$ be i.i.d.\ copies of $(\mathbf{W}, \boldsymbol{\eta}_0)$, 
let $(V^{\eps,i},V^{0,i})$ be the coupled pair \eqref{eq:coupled} driven by $(\mathbf{W}^{i}, \boldsymbol{\eta}_0^{i})$,
and write $F^{\bullet} \triangleq f(V^{\bullet}(T))$, 
$F^{\bullet,i} \triangleq f(V^{\bullet,i}(T))$ for $\bullet\in\{\eps,0\}$, 
and $\mathrm{I}^{0} \triangleq \E[F^{0}]$. 
Following \cite{GM22}, we consider the standard Monte Carlo, 
the control variate and the practical optimal control variate estimators
\begin{equation}\label{eq:estimators}
	\begin{gathered}
		\widehat {\mathrm{I}}^{\eps}_N \triangleq \frac1N\sum_{i=1}^{N}F^{\eps,i},\quad
		\widehat J^{\eps}_N \triangleq \mathrm{I}^{0}+\frac1N\sum_{i=1}^{N}\big(F^{\eps,i}-F^{0,i}\big),\quad
		\widehat K^{\eps}_N \triangleq \widehat\lambda^{\eps}_N \mathrm{I}^{0}+\frac1N\sum_{i=1}^{N}
		\big(F^{\eps,i}-\widehat\lambda^{\eps}_N F^{0,i}\big),
	\end{gathered}
\end{equation}
where $\widehat\lambda^{\eps}_N$ is the empirical regression coefficient of the $F^{\eps,i}$
on the $F^{0,i}$ (set to $0$ when the empirical variance of the $F^{0,i}$ vanishes), which 
estimates $\lambda^{\eps} \triangleq \Cov(F^{\eps},F^{0})/\Var(F^{0})$ ($\lambda^{\eps} \triangleq 0$ if
$\Var(F^{0})=0$). For square-integrable $F^{\eps},F^{0}$ we set
\begin{equation}\label{eq:variances}
	\sigma^{2}_{\mathrm{I}^{\eps}} \triangleq \Var(F^{\eps}),\qquad
	\sigma^{2}_{J^{\eps}} \triangleq \Var(F^{\eps} - F^{0}),\qquad
	\sigma^{2}_{K^{\eps}} \triangleq \min_{\lambda\in\R}\Var\big(F^{\eps}-\lambda F^{0}\big),
\end{equation}
the minimum being attained at $\lambda=\lambda^{\eps}$; in particular
$\sigma^{2}_{K^{\eps}}\le\min\{\sigma^{2}_{\mathrm{I}^{\eps}},\sigma^{2}_{J^{\eps}}\}$ (take
$\lambda\in\{0,1\}$). Let $\rho^{\eps} \triangleq \Cov(F^{\eps},F^{0})/(\sigma_{\mathrm{I}^{\eps}}
\Var(F^{0})^{1/2})$ ($\rho^{\eps} \triangleq 0$ if $\sigma_{\mathrm{I}^{\eps}}\Var(F^{0})=0$). By
\cite[Proposition~2.1]{GM22}, if $f\in\Hol\cup\mathrm{BV}$ and
$\Var(F^{0})>0$, then $\widehat{\mathrm{I}}^{\eps}_N$ and $\widehat J^{\eps}_N$ are unbiased, 
the three estimators satisfy central limit theorems with the asymptotic variances
\eqref{eq:variances}, and
\begin{equation}\label{eq:sigmaK}
	\sigma^{2}_{K^{\eps}}=\big(1-(\rho^{\eps})^{2}\big)\,\sigma^{2}_{\mathrm{I}^{\eps}}
	=\sigma^{2}_{J^{\eps}}-\Var(F^{0})\,(1-\lambda^{\eps})^{2}.
\end{equation}
The control variates target $\mathrm{I}^{\eps}$ whatever the size of 
the bias $\mathrm{I}^{\eps}-\mathrm{I}^{0}$ of the white-noise approximation (\Cref{cor:bias}). 
For terminal observables, $\mathrm{I}^{0}=\int_\R f\,p_T$, 
where $p_T$ is the density of $V^{0}(T)$ (\Cref{prop:density}), 
known in closed form in regime (I) {\cite{CD61,KS84,TVJ10}}
and {, in regimes (II)--(III),} as a spectral series {\cite{TVJ10} and through its Laplace transform \cite{TPJ12}}; 
thus $\mathrm{I}^{0}$ is a one-dimensional quadrature (Section~\ref{sec:ref-values}).

\section{Main results}\label{sec:main}
Throughout this section, Assumptions  \ref{ass:vel}, \ref{ass:colo}, and \ref{ass:nondeg} are assumed to hold. 
For an observable $f$, we set $D^{\eps}\triangleq f\big(V^{\eps}(T)\big)-f\big(V^{0}(T)\big)$, 
we write $Z^{\eps}\triangleq Z^{\eps}(T)=V^{\eps}(T)-V^{0}(T)$ for the terminal coupling error,
and we set $\vth\triangleq\eps(1+\log(1/\eps))^{1/2}$ for $\eps\in(0,1]$. 
We use the mean-square coupling bound of \cite[Lemma~6.1]{GM22}: there are $C_1>0$ and
$\eps_{\rm GM}\in(0,T]$ such that
\begin{equation}\label{eq:C1}
	\E\big[|Z^{\eps}|^{2}\big]\le C_1\eps^{2},\qquad \eps\in(0,\eps_{\rm GM}],
\end{equation}
and we set $\eps_0\triangleq\eps_{\rm GM}\wedge1$. 
Constants denoted by $C$ may depend on 
$\Delta,\gamma,\ell,\sigma,\mathbf{A},\mathbf{K},\mathbf{r},d,d',T$ and $v_0$, 
but not on $\eps$ or on the observable unless this is indicated.

\subsection{Variance rates}\label{sec:thm-main}
The proof of the following theorem is given in \Cref{sec:proof-main}. 
\begin{theorem}\label{thm:main}
	Let $f$ belong to at least one of the classes in {\rm(a)--(d)} below. 
	Then, for every $\eps>0$,
	\begin{equation}\label{eq:main}
		\sigma^{2}_{K^{\eps}}\le\sigma^{2}_{J^{\eps}}\le\E\big[(D^{\eps})^{2}\big],
	\end{equation}
	and, for all $\eps\in(0,\eps_0]$, $\E[(D^{\eps})^{2}]$ satisfies the bound of each class to which $f$ belongs:
	\begin{enumerate}[label={\rm(\alph*)},itemsep=1pt,topsep=2pt]
		\item if $f\in\Hol$ for some $\alpha\in(0,1]$, then
		$\E[(D^{\eps})^{2}]\le[f]_\alpha^{2}\,C_1^{\alpha}\,\eps^{2\alpha}$;
		\item if $f\in\BVJ$, then $\E[(D^{\eps})^{2}]\le C\,\TV(f)^{2}\,\vth$;
		\item if $f\in\Hol\cap\BVJ$, then $\E[(D^{\eps})^{2}]\le C\,[f]_\alpha\TV(f)\,\vth^{1+\alpha}$;
		\item if $f\in\Cusp$ for a finite nonempty set $S$ and some $\alpha\in(0,1]$, then
		\[
		\E\big[(D^{\eps})^{2}\big]\le C_{\alpha,S}\,[f]_{\alpha,S}^{2}\,\psi_\alpha(\eps),
		\qquad
		\psi_\alpha(\eps) \triangleq 
		\begin{cases}
			\vth^{\,1+2\alpha}, & \alpha<1/2,\\[1pt]
			\vth^{\,2}\,\big(1+\log(1/\eps)\big), & \alpha=1/2,\\[1pt]
			\vth^{\,2}, & \alpha>1/2.
		\end{cases}
		\]
	\end{enumerate}
	The constants $C$ and $C_{\alpha,S}$ do not depend on $\eps$ or $f$; 
	the constant in {\rm (c)} depends on $\alpha$, and $C_{\alpha,S}$ depends on $\alpha$ and $|S|$. 
	Part~{\rm(a)} holds without {\rm\Cref{ass:nondeg}}, 
	and parts {\rm(b)--(d)} hold for all $\eps\in(0,1]$. 
\end{theorem}

\subsection{Consequences}\label{sec:consequences}
The proofs of the following corollaries are given in \Cref{pf:coros}.
\begin{corollary}[Bias of the white-noise approximation]\label{cor:bias}
	Always $|\mathrm{I}^{\eps}-\mathrm{I}^{0}|\le\E[|D^{\eps}|]$. Moreover:
	\begin{enumerate}[label= \rm{(\alph*)},itemsep=1pt,topsep=2pt]
		\item if $f\in\Hol$, then $|\mathrm{I}^{\eps}-\mathrm{I}^{0}|\le[f]_\alpha C_1^{\alpha/2}\eps^{\alpha}$ for
		$\eps\in(0,\eps_0]$;
		\item if $f\in\BVJ$, then $|\mathrm{I}^{\eps}-\mathrm{I}^{0}|\le C\,\TV(f)\,\vth$ for $\eps\in(0,1]$;
		\item if $f\in\Cusp$, then $|\mathrm{I}^{\eps}-\mathrm{I}^{0}|\le C'_{\alpha,S}\,[f]_{\alpha,S}\,\vth$ for
		$\eps\in(0,1]$, where $C'_{\alpha,S}$ depends only on $\alpha$, $|S|$, and the model parameters.
	\end{enumerate}
\end{corollary}

Part (b) improves on the $O(\vth^{1/2})$ given by Theorem~\ref{thm:main}(b) and the
Cauchy--Schwarz inequality, because $|D^{\eps}|$ itself is controlled by the crossing
inequality (Lemma~\ref{lem:crossing}); for smooth observables the bias is $O(\eps^{2})$
\cite[Lemma~6.1]{GM22}.

\begin{corollary}[Variance reduction]\label{cor:vr}
	Let $f$ be as in \rm{\Cref{thm:main}} with $\Var(F^{0})>0$, 
	let $\mathcal R_f(\eps)$ be one of the bounds on $\E[(D^{\eps})^{2}]$ given there, 
	and write $\sigma_0 \triangleq \Var(F^{0})^{1/2}$. 
	Then $|\sigma_{\mathrm{I}^{\eps}}-\sigma_0|\le\sigma_{J^{\eps}}$, 
	so that $\sigma_{\mathrm{I}^{\eps}}\to\sigma_0$,
	and, for all $\eps\in(0,\eps_0]$ such that $\mathcal R_f(\eps)\le\sigma_0^{2}/4$,
	\begin{equation*}\label{eq:vr}
		\frac{\sigma^{2}_{K^{\eps}}}{\sigma^{2}_{\mathrm{I}^{\eps}}}=1-(\rho^{\eps})^{2}
		\le\frac{\sigma^{2}_{J^{\eps}}}{\sigma^{2}_{\mathrm{I}^{\eps}}}\le\frac{4\mathcal R_f(\eps)}{\sigma_0^{2}},
		\qquad |1-\lambda^{\eps}|\le\frac{\sigma_{J^{\eps}}}{\sigma_0}\le\frac{\mathcal R_f(\eps)^{1/2}}{\sigma_0}.
	\end{equation*}
\end{corollary}

Since a coupled path costs about twice a colored-noise path and $\mathrm{I}^{0}$ is computed once,
$\widehat J^{\eps}_N$ reaches a given accuracy with about
$2\sigma^{2}_{J^{\eps}}/\sigma^{2}_{\mathrm{I}^{\eps}}=O(\mathcal R_f(\eps))$ times the work of plain
Monte Carlo (see \cite[Section~8]{GM22} when $\mathrm{I}^{0}$ is itself estimated).

\begin{corollary}[Non-asymptotic sample sizes]\label{cor:bernstein}
	Let $f$ be bounded with oscillation $\osc(f) \triangleq \sup f-\inf f$, and let $t>0$,
	$\varrho\in(0,1)$. Then
	\begin{equation}\label{eq:bernstein}
		\PP\big(|\widehat J^{\eps}_N-\mathrm{I}^{\eps}|\ge t\big)\le
		2\exp\Big(-\frac{Nt^{2}}{2\sigma^{2}_{J^{\eps}}+\frac43\osc(f)\,t}\Big).
	\end{equation}
	Hence $|\widehat J^{\eps}_N-\mathrm{I}^{\eps}|<t$ with probability at least $1-\varrho$ as soon as
	\begin{equation}\label{eq:NJ}
		N\ge\Big(\frac{2\sigma^{2}_{J^{\eps}}}{t^{2}}+\frac{4\osc(f)}{3t}\Big)\log\frac2\varrho,
		\qquad\text{against}\qquad
		N\ge\Big(\frac{2\sigma^{2}_{\mathrm{I}^{\eps}}}{t^{2}}+\frac{2\osc(f)}{3t}\Big)\log\frac2\varrho
	\end{equation}
	for $\widehat{ \mathrm{I}}^{\eps}_N$. 
	Under the assumptions of   \rm{\Cref{thm:main}}, 
	the leading term $2\sigma^{2}_{J^{\eps}}t^{-2}$ of the first bound is at most $2\mathcal R_f(\eps)t^{-2}$, 
	with $\mathcal R_f$ as in \rm{\Cref{cor:vr}}, 
	whereas that of the second tends to $2\Var(F^{0})\,t^{-2}$.
\end{corollary}

\section{Proofs of the main results}\label{sec:proofs} 
In this section, Assumptions  \ref{ass:vel}, \ref{ass:colo}, and \ref{ass:nondeg} are assumed to hold, 
and $D^{\eps}$, $Z^{\eps}$, $\vth$ and $\eps_0$ are as in Section~\ref{sec:main}. 
Besides \eqref{eq:C1}, we use the bound $M$ on the density of $V^{0}(T)$ given by \Cref{prop:density}, 
and the Gaussian tails of the coupling error at the scale $\vth$ (\Cref{cor:coupling}). 
For
$u\in\R$ let
\[
E_u \triangleq \big\{u\in\seg{V^{0}(T)}{V^{\eps}(T)}\big\}\subset\big\{|V^{0}(T)-u|\le|Z^{\eps}|\big\};
\]
the map $(\omega,u)\mapsto\1_{E_u}(\omega)$ is $\mathcal F\otimes\mathcal{B}(\R)$-measurable,
since $E_u=\{V^{0}(T)\wedge V^{\eps}(T)\le u\le V^{0}(T)\vee V^{\eps}(T)\}$. For every
$\theta>0$, $E_u\subset\{|Z^{\eps}|>\theta\}\cup\{|V^{0}(T)-u|\le\theta\}$, and therefore,
by Proposition~\ref{prop:density},
\begin{equation}\label{eq:Eu-split}
	\PP(E_u)\le\PP\big(|Z^{\eps}|>\theta\big)+2M\theta,\qquad u\in\R,\ \theta>0 .
\end{equation}
If $f\in\BVJ$, 
the crossing inequality $|f(x)-f(y)|\le|\mu_f|(\seg{x}{y})$ of \Cref{lem:crossing} gives the pathwise bounds
\begin{equation}\label{eq:D-level}
	|D^{\eps}|\le\int_\R\1_{E_u}\,|\mu_f|(\dd u),\qquad
	(D^{\eps})^{2}\le\TV(f)\int_\R\1_{E_u}\,|\mu_f|(\dd u),
\end{equation}
and, by Tonelli's theorem and $|\mu_f|(\R)\le\TV(f)$,
\begin{equation}\label{eq:D-level-exp}
	\E[|D^{\eps}|]\le\TV(f)\sup_{u\in\R}\PP(E_u),\qquad
	\E\big[(D^{\eps})^{2}\big]\le\TV(f)^{2}\sup_{u\in\R}\PP(E_u).
\end{equation}

\subsection{Proof of Theorem~\ref{thm:main}}\label{sec:proof-main}
The inequalities \eqref{eq:main} follow from \eqref{eq:variances}, with $\lambda=1$, and
$\Var(D^{\eps})\le\E[(D^{\eps})^{2}]$. Recall that $\eps\mapsto\vth$ is increasing on
$(0,1]$ with $\vartheta_1=1$, so that $\eps\le\vth\le1$ for $\eps\in(0,1]$. For such $\eps$
and all $p,q>0$, \Cref{cor:coupling}(iii)--(iv) gives, since
$|Z^{\eps}|\le\sup_{t\le T}|Z^{\eps}(t)|$,
\begin{equation}\label{eq:Z-moments}
	\E[|Z^{\eps}|^{p}]\le\mathsf C_p\vth^{p},\qquad
	\PP\big(|Z^{\eps}|>\mathsf K_q\vth\big)\le(T+1)\,C_0\,\eps^{q}.
\end{equation}

(a) Since $|D^{\eps}|\le[f]_\alpha|Z^{\eps}|^{\alpha}$ and $u\mapsto u^{\alpha}$ is concave on
$[0,\infty)$, Jensen's inequality and \eqref{eq:C1} give, for $\eps\in(0,\eps_{\rm GM}]$,
$\E[(D^{\eps})^{2}]\le[f]_\alpha^{2}(\E[|Z^{\eps}|^{2}])^{\alpha}\le[f]_\alpha^{2}C_1^{\alpha}\eps^{2\alpha}$.
Neither \Cref{ass:nondeg} nor Proposition~\ref{prop:density} is used.

(b) Take $\theta=\mathsf K_1\vth$ in \eqref{eq:Eu-split}. By \eqref{eq:Z-moments},
$\PP(E_u)\le(T+1)C_0\eps+2M\mathsf K_1\vth\le\big((T+1)C_0+2M\mathsf K_1\big)\vth$ for all
$u$, and \eqref{eq:D-level-exp} gives the claim.

(c) Combining $|D^{\eps}|\le[f]_\alpha|Z^{\eps}|^{\alpha}$ with \eqref{eq:D-level},
$(D^{\eps})^{2}=|D^{\eps}|\,|D^{\eps}|\le[f]_\alpha|Z^{\eps}|^{\alpha}\int_\R\1_{E_u}|\mu_f|(\dd u)$,
so that, by Tonelli's theorem,
$\E[(D^{\eps})^{2}]\le[f]_\alpha\TV(f)\sup_{u}\E\big[|Z^{\eps}|^{\alpha}\1_{E_u}\big]$.
Let $\theta=\mathsf K_2\vth$. On $E_u\cap\{|Z^{\eps}|\le\theta\}$ we have
$|V^{0}(T)-u|\le\theta$; hence, using the Cauchy--Schwarz inequality on
$\{|Z^{\eps}|>\theta\}$,
\[
\begin{aligned}
	\E\big[|Z^{\eps}|^{\alpha}\1_{E_u}\big]&\le\theta^{\alpha}\,\PP\big(|V^{0}(T)-u|\le\theta\big)
	+\big(\E[|Z^{\eps}|^{2\alpha}]\big)^{1/2}\PP\big(|Z^{\eps}|>\theta\big)^{1/2}\\
	&\le2M\mathsf K_2^{1+\alpha}\vth^{1+\alpha}+\mathsf C_{2\alpha}^{1/2}\vth^{\alpha}\big((T+1)C_0\big)^{1/2}\eps ,
\end{aligned}
\]
by \eqref{eq:Z-moments} with $q=2$, which is bounded by $C\vth^{1+\alpha}$ because
$\eps\le\vth$.

(d) Let $f\in\Cusp$, $\theta \triangleq \mathsf K_4\vth$ and $\Upsilon \triangleq \dist(V^{0}(T),S)$. By
Proposition~\ref{prop:density}, for $x\ge0$ and $h>0$,
$\PP(\Upsilon\in[x,x+h])\le\sum_{\varsigma\in S}\PP(|V^{0}(T)-\varsigma|\in[x,x+h])\le2|S|Mh$,
i.e., the law of $\Upsilon$ has a density bounded by $2|S|M$. We split
$\E[(D^{\eps})^{2}]=\mathrm{(i)}+\mathrm{(ii)}+\mathrm{(iii)}$ according to the events
$\{|Z^{\eps}|>\theta\}$, $\{|Z^{\eps}|\le\theta,\ \Upsilon\le2\theta\}$, and
$\{|Z^{\eps}|\le\theta,\ \Upsilon>2\theta\}$. By Lemma~\ref{lem:cusp}(ii),
$|D^{\eps}|\le c_{\alpha,S}[f]_{\alpha,S}|Z^{\eps}|^{\alpha}$. Hence, by the Cauchy--Schwarz
inequality and \eqref{eq:Z-moments} with $q=4$,
\[
\mathrm{(i)}\le c_{\alpha,S}^{2}[f]_{\alpha,S}^{2}\big(\E[|Z^{\eps}|^{4\alpha}]\big)^{1/2}
\PP(|Z^{\eps}|>\theta)^{1/2}\le C\,c_{\alpha,S}^{2}[f]_{\alpha,S}^{2}\,\vth^{2\alpha}\eps^{2}
\le C\,c_{\alpha,S}^{2}[f]_{\alpha,S}^{2}\,\vth^{2+2\alpha},
\]
and
$\mathrm{(ii)}\le c_{\alpha,S}^{2}[f]_{\alpha,S}^{2}\,\theta^{2\alpha}\,\PP(\Upsilon\le2\theta)
\le4|S|M\,c_{\alpha,S}^{2}[f]_{\alpha,S}^{2}\,\theta^{1+2\alpha}$.
On the event of (iii), every point $u$ of the segment $\seg{V^{0}(T)}{V^{\eps}(T)}$
satisfies $\dist(u,S)\ge\Upsilon-|Z^{\eps}|\ge\Upsilon/2$, so \Cref{lem:cusp}(i) gives
$|D^{\eps}|\le[f]_{\alpha,S}(\Upsilon/2)^{\alpha-1}|Z^{\eps}|\le2[f]_{\alpha,S}\Upsilon^{\alpha-1}\theta$.
Since $\Upsilon^{2\alpha-2}\le1$ for $\Upsilon\ge1$,
\[
\mathrm{(iii)}\le4[f]_{\alpha,S}^{2}\,\theta^{2}\,\E\big[\Upsilon^{2\alpha-2}\1\{\Upsilon>2\theta\}\big]
\le4[f]_{\alpha,S}^{2}\,\theta^{2}\Big(1+2|S|M\int_{2\theta\wedge1}^{1}x^{2\alpha-2}\dd x\Big).
\]
The last integral vanishes if $2\theta\ge1$; otherwise it is at most
$(2\theta)^{2\alpha-1}/(1-2\alpha)$ if $\alpha<1/2$, $\log(1/(2\theta))$ if $\alpha=1/2$, and
$1/(2\alpha-1)$ if $\alpha>1/2$. Recall that $\theta=\mathsf K_4\vth$ with $\eps\le\vth\le1$,
so that $\log(1/(2\theta))\le\log_+(1/(2\mathsf K_4))+\log(1/\eps)$. Collecting the three
terms gives $\E[(D^{\eps})^{2}]\le C_{\alpha,S}[f]_{\alpha,S}^{2}\psi_\alpha(\eps)$: for
$\alpha<1/2$ we use $\theta^{2}(2\theta)^{2\alpha-1}\le\theta^{1+2\alpha}$ and
$\vth^{2}\le\vth^{1+2\alpha}$, and for $\alpha\ge1/2$ we use $\vth^{1+2\alpha}\le\vth^{2}$.
Only $\eps\in(0,1]$ was used in (b)--(d). \qed

\subsection{Proofs of the corollaries}\label{pf:coros}
\begin{proof}[Proof of \Cref{cor:bias}]
	$|\mathrm{I}^{\eps}-\mathrm{I}^{0}|=|\E [D^{\eps}]|\le\E[|D^{\eps}|]$. 
	(a) By Jensen's inequality and \eqref{eq:C1},
	$\E[|D^{\eps}|] \le [f]_\alpha\E[|Z^{\eps}|^{\alpha}] \le [f]_\alpha(\E[|Z^{\eps}|^{2}])^{\alpha/2}
	\le[f]_\alpha C_1^{\alpha/2}\eps^{\alpha}$.
	(b) By \eqref{eq:D-level-exp}, $\E[|D^{\eps}|]\le\TV(f)\sup_u\PP(E_u)$, and the bound
	$\sup_u\PP(E_u)\le((T+1)C_0+2M\mathsf K_1)\vth$ obtained in the proof of
	\Cref{thm:main}(b) gives the claim. 
	(c) Split $\E[|D^{\eps}|]$ according to the
	same three events as in the proof of \Cref{thm:main}(d). The three contributions
	are bounded, respectively, by
	\[
	c_{\alpha,S}[f]_{\alpha,S}\big(\E[|Z^{\eps}|^{2\alpha}]\big)^{1/2}\PP(|Z^{\eps}|>\theta)^{1/2}
	\le C\,c_{\alpha,S}[f]_{\alpha,S}\,\vth^{\alpha}\eps^{2},\qquad
	c_{\alpha,S}[f]_{\alpha,S}\,\theta^{\alpha}\cdot4|S|M\theta,
	\]
	and
	$2[f]_{\alpha,S}\,\theta\,\E[\Upsilon^{\alpha-1}\1_{\{\Upsilon>2\theta\}}]\le2[f]_{\alpha,S}\,\theta\,
	(1+2|S|M/\alpha)$, because $\int_0^{1}x^{\alpha-1}\dd x=1/\alpha$. Since $\eps\le\vth\le1$
	and $\theta=\mathsf K_4\vth$, the sum is at most $C'_{\alpha,S}[f]_{\alpha,S}\vth$.
\end{proof}

\begin{proof}[Proof of Corollary~\ref{cor:vr}]
	By Minkowski's inequality in $L^{2}$,
	$|\sigma_{\mathrm{I}^{\eps}}-\sigma_0|=\big|\|F^{\eps}-\E [F^{\eps}]\|_2-\|F^{0}-\E [F^{0}]\|_2\big|
	\le\|D^{\eps}-\E [D^{\eps}]\|_2=\sigma_{J^{\eps}}$. If $\mathcal R_f(\eps)\le\sigma_0^{2}/4$,
	then $\sigma_{J^{\eps}}\le\sigma_0/2$ by \Cref{thm:main}, 
	hence $\sigma_{\mathrm{I}^{\eps}}\ge\sigma_0/2$ and
	$\sigma^{2}_{J^{\eps}}/\sigma^{2}_{\mathrm{I}^{\eps}}\le4\mathcal R_f(\eps)/\sigma_0^{2}$. 
	The identity $\sigma^{2}_{K^{\eps}}/\sigma^{2}_{\mathrm{I}^{\eps}}=1-(\rho^{\eps})^{2}$ is \eqref{eq:sigmaK}, 
	the inequality $\sigma^{2}_{K^{\eps}}\le\sigma^{2}_{J^{\eps}}$ follows from \eqref{eq:variances},
	and \eqref{eq:sigmaK} also gives
	$\sigma_0^{2}(1-\lambda^{\eps})^{2}=\sigma^{2}_{J^{\eps}}-\sigma^{2}_{K^{\eps}}\le\sigma^{2}_{J^{\eps}}$.
\end{proof}

\begin{proof}[Proof of Corollary~\ref{cor:bernstein}]
	The summands $D^{\eps,i}=F^{\eps,i}-F^{0,i}$ of $\widehat J^{\eps}_N-\mathrm{I}^{0}$ are i.i.d.\ with
	variance $\sigma^{2}_{J^{\eps}}$ and take values in $[-\osc(f),\osc(f)]$, 
	so that $|D^{\eps,i}-\E [D^{\eps}]|\le2\osc(f)$. 
	For averages of $N$ i.i.d.\ variables with variance
	$\sigma^{2}_{J^{\eps}}$ and centered values bounded by $2\osc(f)$, 
	Bernstein's inequality \cite[Chapter~2]{BLM13} is \eqref{eq:bernstein}, 
	whose right-hand side is at most $\varrho$ under the first condition in \eqref{eq:NJ}. 
	For $\widehat {\mathrm{I}}^{\eps}_N$ the summands lie in
	$[\inf f,\sup f]$, so Bernstein's inequality holds with $\sigma^{2}_{\mathrm{I}^{\eps}}$ in place of
	$\sigma^{2}_{J^{\eps}}$ and $\osc(f)$ in place of $2\osc(f)$, 
	which gives the second condition.
\end{proof}

\section{Numerical experiments}\label{sec:numerics}
\subsection{Test problems and predicted exponents}\label{sec:testproblems} 
To test the exponents of Section~\ref{sec:main},
measure the variance reduction, and check
that time discretization does not affect these conclusions,
we take $\sigma=\sqrt2$, $v_0=0.5$, $T=1$, $v_f=2$, $\Delta=1$, $\gamma=1$ in regimes
(II)--(III), $\ell=0.5$ in regime (III) (a stick-phase value),
and the scalar noise $d=d'=1$,
$\mathbf{A}=\mathbf{K}=\mathbf{r}=1$;
then $\mathfrak{s}=\sigma=\sqrt2$, \Cref{ass:nondeg} holds,
and all the scaling of \cite{TVJ10} reduce to
the identity.
The observables are $f_{\rm ind}$, $f_\alpha$ for $\alpha\in\{1/4,1/3,1/2,3/4,1\}$,
and $f_{\rm sm}$ (Example~\ref{ex:observables}), the seven test observables, and, in
	Figure~\ref{fig:rates}(b), $f_\alpha$ for $\alpha\in\{0.1,0.2,0.3,0.4,0.6,0.7,0.8,0.9\}$.
By \Cref{thm:main},
the exponent of $\mathcal D_2(\eps) \triangleq \E[(D^{\eps})^{2}]$ is, up to logarithms,
at least $1$ for $f_{\rm ind}$ and $\min\{2,1+2\alpha\}$ for $f_\alpha$, and $2$ for $f_1$ and $f_{\rm sm}$.
For $f_{\rm ind}$ and for $f_\alpha$ with $\alpha<1$, these exponents exceed the
benchmark $\max\{2\alpha,2/3\}$ given by \eqref{eq:C1} and Avikainen's inequality \cite{Avi09},
so that the experiment separates the two. They are exponents of upper bounds, but we expect
them to be attained, because $|Z^{\eps}(T)|$ is of exact order $\eps$ and $p_T(\pm v_f)>0$.
For the coupling error, \eqref{eq:C1} gives the exponent $2$ for $\E[|Z^{\eps}(T)|^{2}]$, and
\Cref{cor:coupling}(iii) the exponent $2$ up to a logarithm for
$\E[\sup_{t\le T}|Z^{\eps}(t)|^{2}]$.
Smaller measured exponents would indicate a pre-asymptotic regime or an error.

\subsection{Coupled time discretization}\label{sec:scheme}
Let $\delta t=T/n$ and let $\boldsymbol{\zeta}_0,\boldsymbol{\zeta}_1,\dots$ be i.i.d.\ $\mathcal N(0,I_{d'})$ vectors,
independent of $\boldsymbol{\eta}_0\sim\mathcal N(0,\mathbf C)$. 
The OU process is advanced by the Euler scheme, 
and both forcing increments are built from the same Gaussian vectors:
\begin{equation*}\label{eq:OU-discrete}
	\boldsymbol{\eta}_{m+1}
	=\boldsymbol{\eta}_m-\frac{\delta t}{\eps^{2}}\mathbf{A}\boldsymbol{\eta}_m+\frac{\sqrt{\delta t}}{\eps}\mathbf{K}\boldsymbol{\zeta}_m,\qquad
	\delta N^{\eps}_m \triangleq \frac{\sigma\,\delta t}{\eps}\,\mathbf{r}\cdot\boldsymbol{\eta}_m,\qquad
	\delta N^{0}_m \triangleq \sigma\sqrt{\delta t}\,\mathbf{c}^{\top}\mathbf{K}\boldsymbol{\zeta}_m .
\end{equation*}
Multiplying the first relation by $\eps\sigma \mathbf{r}^{\top}\mathbf{A}^{-1}$ gives
$\delta N^{\eps}_m=\delta N^{0}_m-\eps \sigma \mathbf{c}^{\top}(\boldsymbol{\eta}_{m+1}-\boldsymbol{\eta}_m)$, 
hence the exact discrete analogue of \eqref{eq:integrated},
\begin{equation}\label{eq:discrete-identity}
	h_m \triangleq \sum_{j<m}\big(\delta N^{\eps}_j-\delta N^{0}_j\big)=-\sigma\eps\,\mathbf{c}^{\top}(\boldsymbol{\eta}_m-\boldsymbol{\eta}_0),
	\qquad m\ge0 .
\end{equation}
The velocities are advanced by the backward Euler (proximal) step for $U$,
\begin{equation}\label{eq:BE}
	V_{m+1}=\operatorname{prox}_{\delta t\,U}\big(V_m+\delta N_m\big)
	=\frac{\operatorname{soft}_{\Delta\delta t}\big(V_m+\ell\,\delta t+\delta N_m\big)}{1+\gamma\,\delta t},
	\qquad \operatorname{soft}_{\tau}(x) \triangleq \sgn(x)\,(|x|-\tau)_+,
\end{equation}
with $\delta N_m=\delta N^{\eps}_m$ for $V^{\eps}$, $\delta N_m=\delta N^{0}_m$ for $V^{0}$, and $V_0=v_0$; 
equivalently, $V_{m+1}=V_m+\delta N_m+\delta t\,(\bb(V_{m+1})-k_{m+1})$ with $k_{m+1}\in\partial\varphi(V_{m+1})$. 
For $V^{0}$ this is the backward Euler--Maruyama method for the stochastic gradient flow of the convex potential $U$ \cite{EKKL22}; 
for $V^{\eps}$ it is the implicit Euler scheme for a maximal monotone differential inclusion \cite{Bas13}. 
We take $n=\lceil T/\min\{10^{-3},\eps^{2}/50\}\rceil$, so that $n\,\delta t=T$ and
$\delta t\,\mathbf{A}/\eps^{2}\le1/50$ for every $\eps$ in the test problem. 
The algorithm, this choice of $\delta t$, and a semi-implicit variant to which \Cref{prop:discrete} also applies are
detailed in Appendix~\ref{app:implementation}.

\begin{proposition}[Discrete comparison principle]\label{prop:discrete}
	Let $(V^{\eps}_m)$ and $(V^{0}_m)$ be generated by \eqref{eq:BE} from the same
	$(\boldsymbol{\eta}_0,\boldsymbol{\zeta}_0,\boldsymbol{\zeta}_1,\dots)$. Then, for every $n\ge0$,
	\begin{equation}\label{eq:discrete-comparison}
		\max_{m\le n}|V^{\eps}_m-V^{0}_m|\le2\max_{m\le n}|h_m|\le4\kappa_0\,\eps\max_{m\le n}|\boldsymbol{\eta}_m| .
	\end{equation}
\end{proposition}

\begin{proof}
	Let $N_m \triangleq \sum_{j<m}\delta N_j$ for each process, $Y_m \triangleq V_m-N_m$,
	$\Lambda_m \triangleq Y^{\eps}_m-Y^{0}_m$, so that $\Lambda_0=0$ and
	$V^{\eps}_m-V^{0}_m=\Lambda_m+h_m$, and let $H_m \triangleq \max_{j\le m}|h_j|$. 
	We show by induction
	that $|\Lambda_m|\le H_m$, which gives the first inequality in
	\eqref{eq:discrete-comparison}; the second follows from \eqref{eq:discrete-identity}. 
	Here $Y_{m+1}=Y_m-\delta t\,g_{m+1}$ with $g_{m+1} \triangleq k_{m+1}-\bb(V_{m+1})\in\partial U(V_{m+1})$.
	Assume $|\Lambda_m|\le H_m$ and suppose $\Lambda_{m+1}>H_{m+1}$. 
	Then $V^{\eps}_{m+1}-V^{0}_{m+1}=\Lambda_{m+1}+h_{m+1}>0$, 
	so $g^{\eps}_{m+1}\ge g^{0}_{m+1}$ by monotonicity of $\partial U$, and
	$\Lambda_{m+1}=\Lambda_m-\delta t(g^{\eps}_{m+1}-g^{0}_{m+1})\le\Lambda_m\le H_m\le H_{m+1}$, 
	a contradiction. Symmetrically $\Lambda_{m+1}\ge-H_{m+1}$.
\end{proof}

Thus the coupling error of the simulated pair is bounded uniformly in $\delta t$; the strong
discretization error of each component (only known to be of order at least $\delta t^{1/4}$
for $V^{0}$ \cite{EKKL22}) enters only the bias of the estimators, which the refinement test
below monitors. 
The first inequality in \eqref{eq:discrete-comparison} is checked on every simulated path.

\subsection{Reference values and statistical methodology}\label{sec:ref-values}
\paragraph{Reference values.} 
The value $\mathrm{I}^{0}=\int_\R f\,p_T$ is computed by adaptive quadrature from the closed form \eqref{eq:pI} in regime (I) 
and from the spectral series of \cite{TVJ10}, truncated at $40$ modes, in regimes (II)--(III). 
This density is cross-validated against the numerical inversion of the Laplace transform of \cite{TPJ12}, 
against a finite-volume solution of the Fokker--Planck equation, 
and by simulation (Appendix~\ref{app:implementation} and Table~\ref{tab:ref}): 
the spectral and Laplace representations agree to $10^{-15}$, 
and the reference values agree with their finite-volume approximations to $2\cdot10^{-7}$. 
The error of $\mathrm{I}^{0}$, a bias of $\widehat J^{\eps}_N$ and $\widehat K^{\eps}_N$, 
is therefore well below the standard errors, which are of order $\sigma_{J^{\eps}}N^{-1/2}$.

\paragraph{Measured quantities.} 
(1) $\E[|Z^{\eps}(T)|^{2}]$ and $\E[\max_{m\le n}|Z^{\eps}_m|^{2}]$,
the discrete counterpart of $\E[\sup_{t\le T}|Z^{\eps}(t)|^{2}]$, 
which test \eqref{eq:C1} and \Cref{cor:coupling}.  
(2) $\mathcal D_2(\eps)$ for each observable, 
with least-squares slopes and local slopes $r_i$ of $\log\mathcal D_2$ against
$\log\eps$ between consecutive grid values (Appendix~\ref{app:implementation}); 
the smooth anchor calibrates the finite-grid deficit shared by all observables. 
(3) The variance ratios $\sigma^{2}_{J^{\eps}}/\sigma^{2}_{\mathrm{I}^{\eps}}$ 
and $\sigma^{2}_{K^{\eps}}/\sigma^{2}_{\mathrm{I}^{\eps}}=1-(\rho^{\eps})^{2}$ 
and $\widehat\lambda^{\eps}_N$ (Corollary~\ref{cor:vr}). 
(4) The coupled mean $\E[D^{\eps}]$, 
an unbiased low-variance estimator of $\mathrm{I}^{\eps}-\mathrm{I}^{0}$ (Corollary~\ref{cor:bias}). 
(5) The ratio $\max_{m}|Z^{\eps}_m|/(2\max_{m}|h_m|)$, at most $1$ by \Cref{prop:discrete}, 
and the standardized difference between the sample mean of $f(V^{0}_n)$ and $\mathrm{I}^{0}$, 
which contains the weak error of the scheme for $V^{0}$.

\paragraph{Grid and sampling.} 
The rate study uses thirteen values of $\eps$: 
from $0.5$ to $0.2$ in steps of $0.05$, and $0.15$, $0.1$, $0.05$, $0.04$, $0.03$, $0.02$, 
supplemented by $\eps\in\{0.6,0.8,1\}$ to delimit the range in which the method is useful. 
Each cell (regime,~$\eps$) uses $N=2.4\cdot10^{5}$ coupled paths if $\eps\ge0.05$ and $N=4.8\cdot10^{5}$ otherwise, 
split into $\mathfrak{B}=12$ batches with independent random streams derived from a fixed master seed. 
Means, $\mathcal D_2$ and the estimators \eqref{eq:estimators} are computed from all $N$
paths, with the usual standard errors; 
variances, $\widehat\lambda^{\eps}_N$ and variance ratios have batch-means standard errors. 
The exponents are weighted least-squares slopes of
$\log\mathcal D_2$ against $\log\eps$ over the four values $\eps\le0.05$; 
those of the bias are fitted over the values $\eps\le 0.2$ at which $|\E [D^{\eps}]|$ exceeds three standard
errors, when there are at least three (Appendix~\ref{app:implementation}). 
The refinement test repeats the cells $\eps\in\{0.2,0.1,0.05\}$ with 
$n=\lceil T/\min\{10^{-3},\eps^{2}/400\}\rceil$ and $N=2.4\cdot10^{5}$ independent paths, 
and reports the differences in units of their combined standard error. 
In total, about $4.8\cdot10^{11}$ coupled path-steps are simulated.

\subsection{Results}\label{sec:results}
\paragraph{Coupling errors.} 
The fitted exponents of $\E[|Z^{\eps}(T)|^{2}]$ are $2.0002(35)$, $1.9970(33)$, 
and $2.001(5)$ in regimes (I)--(III) (Table~\ref{tab:fitted}), 
in agreement with the exponent $2$ of \eqref{eq:C1}. 
Those of $\E[\max_m|Z^{\eps}_m|^{2}]$, $1.760(7)$, $1.746(8)$, and $1.746(7)$, 
are smaller because of the logarithmic factor in \Cref{cor:coupling}(iii): 
the maximum of the OU process over the $T/\eps^{2}$ correlation times grows like
$(\log(1/\eps))^{1/2}$ \cite{GP00}, 
and the local exponent $2-1/(a+\log(1/\eps))$ of $\eps^{2}(a+\log(1/\eps))$, $a\ge0$, 
lies between $1.67$ and $2$ on $[0.02,0.05]$.

\paragraph{Exponents of $\mathcal D_2$.} 
For $f_{\rm ind}$, $f_1$ and $f_{\rm sm}$, the fitted exponents agree with \Cref{thm:main} 
to within $2.0$ standard errors (Table~\ref{tab:fitted}); 
in particular, the exponent for the crossing indicator is $1.08(4)$,
$0.83(10)$, and $1.02(10)$ in regimes (I)--(III). For $f_\alpha$ with $1/4\le\alpha\le3/4$, the
fitted exponents are at most $0.30$ below $\min\{2,1+2\alpha\}$, with the largest deficits at
$\alpha=1/2$ (Figure~\ref{fig:rates}(b)). Such deficits are expected on a finite window.
Heuristically, the first correction to the leading term of $\mathcal D_2(\eps)$ has relative order
$\eps^{|1-2\alpha|}$, which decays slowly for $\alpha$ near $1/2$, and for $\alpha=1/2$ the leading
term carries a logarithm. A model computation, in which $Z^{\eps}(T)$ is replaced by a centred
Gaussian variable independent of $V^{0}(T)$ with the measured variance
(Appendix~\ref{app:implementation}), 
gives on the same window exponents below $\min\{2,1+2\alpha\}$ by $0.04$ to $0.23$ for these $\alpha$, 
and by at most $0.02$ for $f_{\rm ind}$, $f_1$ and $f_{\rm sm}$; 
the fitted exponents of all seven observables agree with these model exponents to within $2.3$ standard errors. 
Finally, the fitted exponents exceed the benchmark $\max\{2\alpha,2/3\}$ in all $15$ cases in
which the benchmark is below the prediction, 
by more than three standard errors in $14$ of them 
(the exception is $f_{\rm ind}$ in regime~(II), with $1.7$ standard errors): 
$\mathcal D_2$ decays faster than the mean-square bound \eqref{eq:C1} alone can guarantee, as \Cref{thm:main} asserts.

\paragraph{Bias.} 
The coupled bias is resolved for all $21$ pairs of observable and regime,
with fitted exponents between $1.74$ and $2.08$. 
These exceed the exponent $1$ of \Cref{cor:bias} and are close to the exponent $2$ known for smooth observables \cite[Lemma~6.1]{GM22}, 
which suggests that the bias bounds of \Cref{cor:bias}(b)--(c) are not sharp for these examples. 
The white-noise approximation is nevertheless inaccurate at moderate $\eps$: 
for $f_{\rm ind}$ at $\eps=0.2$, 
$|\E [D^{\eps}]|$ is $12$\% to $34$\% of $\mathrm{I}^{\eps}$ (Table~\ref{tab:vr}), 
a bias that the control variate estimators remove.

\paragraph{Variance reduction.} 
For the crossing indicator (Table~\ref{tab:vr}),
$\sigma^{2}_{K^{\eps}}/\sigma^{2}_{\mathrm{I}^{\eps}}$ decreases from $0.99$--$0.998$
at $\eps=1$ to $0.0359$--$0.0588$ at $\eps=0.02$, 
roughly in proportion to $\eps$ for $\eps\le0.1$,  
as \Cref{thm:main}(b) and \Cref{cor:vr} predict, 
while $\widehat\lambda^{\eps}_N$ increases to $0.970$--$0.982$. 
At $\eps=0.02$, $\widehat K^{\eps}_N$ is therefore at least $8.51$ times cheaper
than plain Monte Carlo at equal accuracy. 
At $\eps=1$, $F^{\eps}$ and $F^{0}$ are
almost uncorrelated (in regime~(II), no path of $V^{\eps}$ reaches the threshold):
$\widehat J^{\eps}_N$ is useless, 
with $\sigma^{2}_{J^{\eps}}/\sigma^{2}_{\mathrm{I}^{\eps}}$ up to
$355$, whereas $\widehat\lambda^{\eps}_N\approx0$ turns $\widehat K^{\eps}_N$ into plain
Monte Carlo at twice the cost. 
For the smooth anchor,
$\sigma^{2}_{K^{\eps}}/\sigma^{2}_{\mathrm{I}^{\eps}}$ at $\eps=0.02$ is between
$9.4\cdot10^{-4}$ and $1.4\cdot10^{-3}$, 
in accordance with its larger exponent.

\begin{figure}[H]
\centering
\includegraphics[width=1\linewidth]{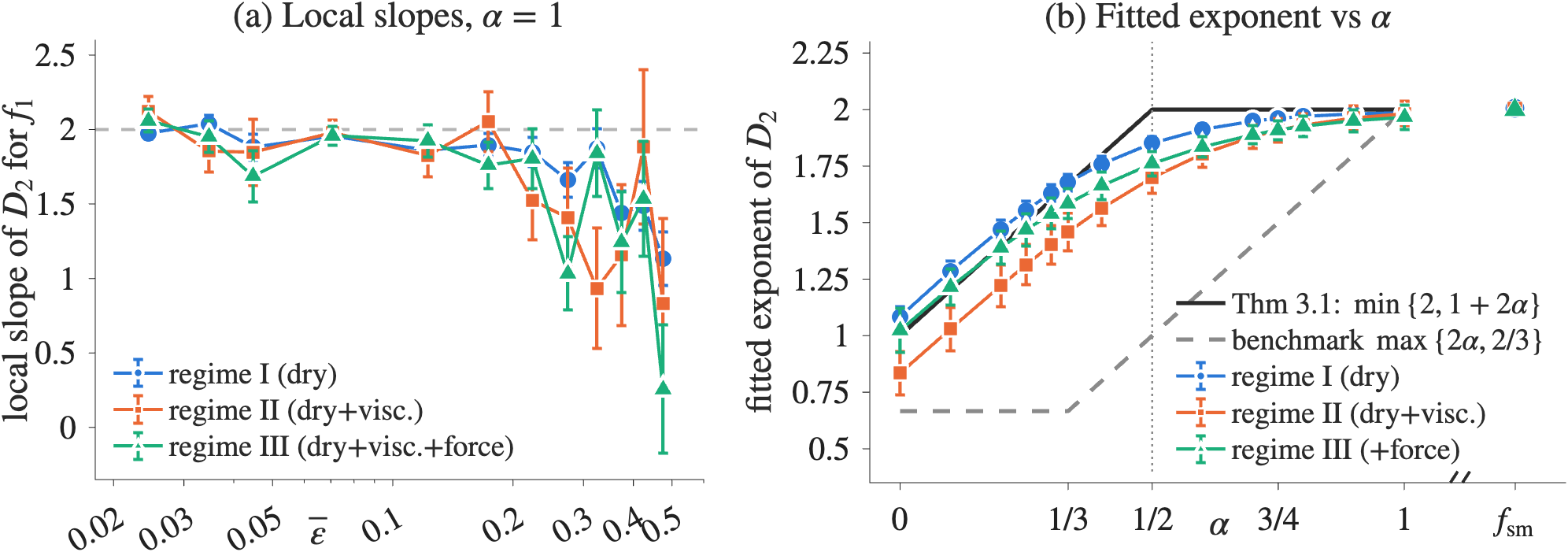}
\caption{(a) Local slopes $r_i$ of $\mathcal D_2$ for $f_1$, plotted at the midpoints
	$\bar\eps_i$ of consecutive grid values, with standard errors, in regimes (I)--(III); 
	dashed: the exponent $2$ of \Cref{thm:main}(a). 
	(b) Fitted exponent of $\mathcal D_2$ over $\eps\le 0.05$, with standard errors, 
	versus $\alpha$ for the severity family, with $f_{\rm ind}$ at $\alpha=0$ and $f_{\rm sm}$ on the right; 
	solid: $\min\{2,1+2\alpha\}$ (Theorem~\ref{thm:main}(b), (d)); 
	dashed: the second-moment benchmark $\max\{2\alpha,2/3\}$ given by \eqref{eq:C1}.} \label{fig:rates}
\end{figure}

%Table 1
\begin{table}[H]
\centering
\caption{Fitted exponents of $\mathcal D_2(\eps)$ (over $\eps\le 0.05$) 
and of the coupled bias $|\E [D^{\eps}]|$ (over the resolved values $\eps\le 0.2$) 
for the seven test observables in regimes (I)--(III), 
against the exponents predicted by \Cref{thm:main} and the benchmark $\max\{2\alpha,2/3\}$, 
and fitted exponents of the coupling errors ($^{*}$:~up to a logarithm). 
For the bias, \Cref{cor:bias} gives the exponent $1$ up to a logarithm, 
and \cite[Lemma~6.1]{GM22} the exponent $2$ for $f_{\rm sm}$. 
Digits in parentheses are standard errors in units of the last digit.}\label{tab:fitted} 
\small
\setlength{\tabcolsep}{4.5pt}
\begin{tabular}{@{}lcc@{\quad}ccc@{\quad}ccc@{}}
	\toprule
	& \multicolumn{2}{c}{prediction} & \multicolumn{3}{c}{exponent of $\mathcal D_2(\eps)$} 
	& \multicolumn{3}{c}{exponent of $|\E [D^{\eps}]|$}\\
	\cmidrule(lr){2-3}\cmidrule(lr){4-6}\cmidrule(l){7-9}
	& \Cref{thm:main} & bench. & (I) & (II) & (III) & (I) & (II) & (III)\\
	\midrule
	$f_{\rm ind}$ & $1$ & $2/3$ & $1.08(4)$ & $0.83(10)$ & $1.02(10)$ & $1.75(13)$ & $1.82(35)$ & $2.1(5)$\\
	$f_{1/4}$ & $3/2$ & $2/3$ & $1.55(4)$ & $1.31(9)$ & $1.47(7)$ & $1.82(6)$ & $1.74(16)$ & $1.89(12)$\\
	$f_{1/3}$ & $5/3$ & $2/3$ & $1.68(4)$ & $1.46(8)$ & $1.58(7)$ & $1.83(6)$ & $1.75(14)$ & $1.92(9)$\\
	$f_{1/2}$ & $2$ & $1$ & $1.851(28)$ & $1.70(7)$ & $1.76(5)$ & $1.85(5)$ & $1.78(10)$ & $1.88(7)$\\
	$f_{3/4}$ & $2$ & $3/2$ & $1.961(20)$ & $1.91(5)$ & $1.91(4)$ & $1.89(5)$ & $1.83(8)$ & $1.87(6)$\\
	$f_{1}$ & $2$ & $2$ & $1.986(19)$ & $1.98(5)$ & $1.97(5)$ & $1.90(4)$ & $1.87(9)$ & $1.86(6)$\\
	$f_{\rm sm}$ & $2$ & $2$ & $2.005(6)$ & $2.002(6)$ & $2.000(8)$ & $1.990(20)$ & $1.995(17)$ & $1.952(14)$\\
	\midrule
	$\E[|Z^{\eps}(T)|^{2}]$ & $2$ & -- & $2.0002(35)$ & $1.9970(33)$ & $2.001(5)$ &  &  & \\
	$\E[\max_{m}|Z^{\eps}_m|^{2}]$ & $2^{*}$ & -- & $1.760(7)$ & $1.746(8)$ & $1.746(7)$ &  &  & \\
	\bottomrule
\end{tabular}
\end{table}

%Table 2
\begin{table}[H]
\centering
\caption{Crossing indicator $f_{\rm ind}$: 
estimate $\widehat K^{\eps}_N$ of $\mathrm{I}^{\eps}$,
coupled estimate of the bias $\E [D^{\eps}]=\mathrm{I}^{\eps}-\mathrm{I}^{0}$ of the white-noise
approximation, 
$\sigma_{\mathrm{I}^{\eps}}$, $\widehat\lambda^{\eps}_N$ and the variance ratios of
Corollary~\ref{cor:vr}, with standard errors in parentheses. 
At equal accuracy,
$\widehat K^{\eps}_N$ costs about $2\sigma^{2}_{K^{\eps}}/\sigma^{2}_{\mathrm{I}^{\eps}}$ times the
work of plain Monte Carlo. 
In regime (II) at $\eps=1$, no path of $V^{\eps}$ reached the threshold,
so that the ratios are undefined (--). }\label{tab:vr}
\small
\begin{tabular}{@{}r@{\quad}cccccc@{}}
	\toprule
	$\eps$ & $\widehat K^{\eps}_N$ & $\E [D^{\eps}]$ & $\sigma_{\mathrm I^{\eps}}$ & $\widehat\lambda^{\eps}_N$ & $\sigma^{2}_{J^{\eps}}/\sigma^{2}_{\mathrm I^{\eps}}$ & $\sigma^{2}_{K^{\eps}}/\sigma^{2}_{\mathrm I^{\eps}}$\\
	\midrule
	\multicolumn{7}{@{}l}{regime (I), \ $\mathrm I^{0}=0.058095$}\\
	$1$ & $0.00193(9)$ & $-0.0557(5)$ & $0.0438$ & $0.0191(13)$ & $28.2(13)$ & $0.9897(11)$\\
	$0.5$ & $0.02606(29)$ & $-0.0319(4)$ & $0.159$ & $0.316(4)$ & $1.793(25)$ & $0.7850(35)$\\
	$0.2$ & $0.05187(26)$ & $-0.00609(28)$ & $0.221$ & $0.7720(34)$ & $0.395(6)$ & $0.338(4)$\\
	$0.1$ & $0.05631(20)$ & $-0.00185(21)$ & $0.231$ & $0.8931(24)$ & $0.1899(30)$ & $0.1782(25)$\\
	$0.05$ & $0.05755(15)$ & $-0.00052(15)$ & $0.232$ & $0.9461(22)$ & $0.1002(31)$ & $0.0973(29)$\\
	$0.02$ & $0.05810(6)$ & $0.00001(6)$ & $0.233$ & $0.9819(4)$ & $0.0362(6)$ & $0.0359(6)$\\
	\midrule
	\multicolumn{7}{@{}l}{regime (II), \ $\mathrm I^{0}=0.007699$}\\
	$1$ & $0$ & $-0.00760(18)$ & $0$ & $0$ & -- & --\\
	$0.5$ & $0.00120(7)$ & $-0.00663(18)$ & $0.0349$ & $0.105(6)$ & $6.05(33)$ & $0.930(7)$\\
	$0.2$ & $0.00579(11)$ & $-0.00195(13)$ & $0.0763$ & $0.601(10)$ & $0.730(24)$ & $0.519(11)$\\
	$0.1$ & $0.00732(9)$ & $-0.00038(9)$ & $0.0852$ & $0.841(5)$ & $0.282(8)$ & $0.256(7)$\\
	$0.05$ & $0.00763(6)$ & $-0.00005(6)$ & $0.0861$ & $0.934(6)$ & $0.126(7)$ & $0.122(6)$\\
	$0.02$ & $0.007689(30)$ & $-0.000006(31)$ & $0.0866$ & $0.9698(24)$ & $0.060(4)$ & $0.059(4)$\\
	\midrule
	\multicolumn{7}{@{}l}{regime (III), \ $\mathrm I^{0}=0.013396$}\\
	$1$ & $0.000037(12)$ & $-0.01347(24)$ & $0.00612$ & $0.0022(6)$ & $355(23)$ & $0.9984(4)$\\
	$0.5$ & $0.00384(12)$ & $-0.00940(22)$ & $0.0616$ & $0.191(7)$ & $3.13(11)$ & $0.875(7)$\\
	$0.2$ & $0.01089(14)$ & $-0.00253(16)$ & $0.104$ & $0.666(8)$ & $0.592(15)$ & $0.455(8)$\\
	$0.1$ & $0.01297(11)$ & $-0.00042(12)$ & $0.113$ & $0.860(8)$ & $0.257(14)$ & $0.236(11)$\\
	$0.05$ & $0.01318(8)$ & $-0.00022(8)$ & $0.114$ & $0.927(5)$ & $0.131(8)$ & $0.126(8)$\\
	$0.02$ & $0.01331(4)$ & $-0.00009(4)$ & $0.115$ & $0.9727(24)$ & $0.049(4)$ & $0.048(4)$\\
	\bottomrule
\end{tabular}
\end{table}

\section{Conclusion}\label{sec:concl}
For the inertial dry-friction inclusion driven by colored noise, 
rate estimates are derived for the control variate estimators of \cite{GM22} with non-smooth observables. 
The analysis rests on two structural properties of the scalar velocity inclusion. 
The first is a pathwise comparison principle, 
which yields \(2\)-Lipschitz continuity of the solution map for the uniform norm and survives time discretization. 
The second is an explicit density bound for the white-noise terminal velocity, 
obtained by a Girsanov comparison with pure dry friction. 
These properties give Gaussian-tail coupling errors and 
yield \(O(\varepsilon^{2\alpha})\) rates for \(\alpha\)-H\"older observables, 
\(O(\varepsilon)\) rates for observables of bounded variation satisfying a one-sided continuity condition, 
\(O(\varepsilon^{1+\alpha})\) rates for observables that are both \(\alpha\)-H\"older and of bounded variation, 
and \(O(\varepsilon^{\min\{2,1+2\alpha\}})\) rates for observables with finitely many \(\alpha\)-cusps. 
The last three rates hold up to logarithmic factors. 
These rate estimates also quantify the bias of the white-noise approximation and 
yield explicit non-asymptotic sample sizes for reliability analysis. 
Moreover, the discrete comparison principle ensures that the coupling remains valid after time discretization. 

Numerical experiments in the three regimes support these rates: 
the fitted exponents of the second moment agree with the theorem for the crossing indicator and the Lipschitz observables,
agree with a model computation of the pre-asymptotic corrections for the severity family, and
exceed those that the mean-square coupling bound alone would give. 
The experiments also indicate that the bias of the white-noise approximation is of order $\varepsilon^{2}$ 
for these non-smooth observables, 
which the bounds proved here do not capture; proving this is left for future work.

\section*{Code and data availability}
No experimental data were used in this study. 
The Python code that produces the numerical results, 
tables and figures of Section~\ref{sec:numerics} and Appendix~\ref{app:implementation}, 
together with the generated data files, 
is available at \url{https://github.com/Lihong-Guo/Dry-friction.git} for research purposes. 

\section*{Use of generative AI}
The author used DeepSeek and Claude (Anthropic) for language editing, 
mathematical assistance, and numerical scripting. 
The author reviewed all outputs and take full responsibility for the final content. 

\bibliographystyle{unsrtnat}%{plain}
\bibliography{References}

%  SUPPLEMENTARY MATERIAL 
\appendix
\begin{center}
	{\Large\bfseries Appendices}
\end{center}

\section{Auxiliary lemmas}\label{app:aux}
\subsection{No sticking under white noise}\label{app:nostick}
The following lemma is a standard consequence of the occupation-times formula
\cite[Chapter~VI, Corollary~1.6]{RY99}. 
The absence of sticking under white noise is noted in \cite{TVJ10}.

\begin{lemma}[No sticking under white noise]\label{lem:nostick}
	Assume that \rm{Assumptions \ref{ass:vel}}, \ref{ass:colo}, and \rm{\ref{ass:nondeg}} hold. 
	Then $V^{0}$ is a continuous semimartingale with $\langle V^{0}\rangle_t=\mathfrak{s}^{2}t$. 
	Almost surely $\Leb\{t\in[0,T]:V^{0}(t)=0\}=0$,
	$k^{0}(t)=\Delta\sgn(V^{0}(t))$ for a.e.\ $t\in[0,T]$, and, with the single-valued drift
	$\drift(v) \triangleq \bb(v)-\Delta\sgn(v)$,
	\begin{equation}\label{eq:white-sde}
		V^{0}(t)=v_0+\int_0^{t}\drift\big(V^{0}(u)\big)\,\dd u+\mathfrak{s}B(t),\qquad t\in[0,T].
	\end{equation}
\end{lemma}

\begin{proof}
	By \eqref{eq:SP-integral}, 
	$V^{0}(t)=v_0+\int_0^{t}(\bb(V^{0})-k^{0})\dd u+\mathfrak{s}B(t)$, 
	where the integral is a continuous adapted process of finite variation (it equals $V^{0}(t)-v_0-\mathfrak{s}B(t)$). 
	Hence $V^{0}$ is a continuous semimartingale with $\langle V^{0}\rangle_t=\mathfrak{s}^{2}t$. 
	By the occupation-times formula,
	$\mathfrak{s}^{2}\int_0^{T}\1_{\{V^{0}(t)=0\}}\dd t=\int_\R\1_{\{0\}}(a)L^{a}_T(V^{0})\,\dd a=0$ almost
	surely, where $L^{a}(V^{0})$ is the local time of $V^{0}$ at level $a$; 
	since $\mathfrak{s}>0$, $\Leb\{t:V^{0}(t)=0\}=0$. 
	For $t$ with $V^{0}(t)\neq0$, $\partial\varphi(V^{0}(t))$ is the
	singleton $\{\Delta\sgn(V^{0}(t))\}$, 
	so $k^{0}(t)=\Delta\sgn(V^{0}(t))$ for a.e.\ $t$, and \eqref{eq:white-sde} follows.
\end{proof}

\subsection{Maximal inequality for the fast OU process}\label{app:OU}
The following two lemmas are standard: 
the expected supremum of an OU process over a time interval of length $t$ grows like $(\log t)^{1/2}$ \cite{GP00}. 
We include their short proofs because the constants $C_0$ and $c_0$ enter \Cref{cor:coupling}.
\begin{lemma}\label{lem:OU-unit}
	Let $\widetilde{\boldsymbol{\eta}}$ be the stationary solution of
	$\dd\widetilde{\boldsymbol{\eta}}
	=-\mathbf{A}\widetilde{\boldsymbol{\eta}}\,\dd u+\mathbf{K}\,\dd\widetilde {\mathbf{W}}$, 
	where $\widetilde{\mathbf{W}}$ is a
	standard $d'$-dimensional Brownian motion and 
	$\widetilde{\boldsymbol{\eta}}(0)\sim\mathcal N(0,\mathbf C)$ is independent of $\widetilde{\mathbf{W}}$. 
	Let $a_A \triangleq \max_{u\in[0,1]}\|e^{-\mathbf{A}u}\|$,
	$b_{A,K} \triangleq \|\mathbf{K}\|+\max_{u\in[0,1]}\|\mathbf{A}e^{-\mathbf{A}u}\mathbf{K}\|$ 
	and $c_{\mathbf C} \triangleq \max_i\mathbf C_{ii}$. 
	Then
	$Y_1 \triangleq \sup_{u\in[0,1]}|\widetilde{\boldsymbol{\eta}}(u)|$ satisfies, for all $y\ge0$,
	\[
	\PP(Y_1>y)\le d\,e^{-y^{2}/(8a_A^{2}d\,c_{\mathbf C})}+2d'\,e^{-y^{2}/(8b_{A,K}^{2}d')}
	\le C_0\,e^{-c_0y^{2}},
	\]
	with $C_0 \triangleq d+2d'$ and 
	$c_0 \triangleq \min\{(8a_A^{2}d\,c_{\mathbf C})^{-1},(8b_{A,K}^{2}d')^{-1}\}$ (if
	$K=0$, then $\widetilde{\boldsymbol{\eta}}\equiv0$ and there is nothing to prove).
\end{lemma}

\begin{proof}
	By the variation-of-constants formula,
	$\widetilde{\boldsymbol{\eta}}(u) 
	= e^{-\mathbf{A}u}\widetilde{\boldsymbol{\eta}}(0)
	+\int_0^{u}e^{-\mathbf{A}(u-r)}\mathbf{K}\,\dd\widetilde{\mathbf{W}}_r$. 
	Since $r\mapsto e^{-\mathbf{A}(u-r)}\mathbf{K}$ is deterministic and $C^{1}$, 
	It\^o's product rule gives
	$\int_0^{u}e^{-\mathbf{A}(u-r)}\mathbf{K}\,\dd\widetilde {\mathbf W}_r
	=\mathbf{K}\widetilde {\mathbf W}_u-\int_0^{u}{\mathbf A}e^{-{\mathbf A}(u-r)}{\mathbf K}\widetilde {\mathbf W}_r\dd r$
	almost surely, for each $u$, and by continuity simultaneously for all $u\in[0,1]$. 
	Hence $Y_1\le a_A|\widetilde{\boldsymbol{\eta}}(0)|+b_{A,K}\sup_{r\le1}|\widetilde {\mathbf W}_r|$. 
	For a centered Gaussian
	vector $X\in\R^{d}$ with covariance $\mathbf C$,
	$\PP(|X|>z)\le\sum_{i}\PP(|X_i|>z/\sqrt d)\le d\,e^{-z^{2}/(2dc_{\mathbf C})}$, using
	$\PP(|\mathcal N(0,\varsigma^{2})|>x)\le e^{-x^{2}/(2\varsigma^{2})}$. 
	By the reflection principle, $\PP(\sup_{r\le1}|\widetilde {\mathbf W}^{j}_r|>x)\le2\PP(|\widetilde {\mathbf W}^{j}_1|>x)\le2e^{-x^{2}/2}$
	for each coordinate, so $\PP(\sup_{r\le1}|\widetilde {\mathbf W}_r|>z)\le2d'e^{-z^{2}/(2d')}$. 
	The claim follows from
	$\PP(Y_1>y)\le\PP(a_A|\widetilde{\boldsymbol{\eta}}(0)|>y/2)
	+\PP(b_{A,K}\sup_{r\le1}|\widetilde {\mathbf W}_r|>y/2)$.
\end{proof}

\begin{lemma}[Maximal inequality for the fast OU process]\label{lem:OUmax}
	Assume that \rm{\Cref{ass:colo}} holds, and let $C_0,c_0$ be as in Lemma~\ref{lem:OU-unit}. 
	For all $\eps>0$, $t>0$ and $y\ge0$,
	\begin{equation}\label{eq:OUmax}
		\PP\Big(\sup_{u\le t}|\boldsymbol{\eta}^{\eps}(u)|>y\Big)\le C_0\,\Big\lceil\frac{t}{\eps^{2}}\Big\rceil\,e^{-c_0y^{2}} .
	\end{equation}
	Consequently, for every $p>0$ there is $\mathsf L_p$ (depending also on $T$) such that
	$\E\big[\sup_{u\le T}|\boldsymbol{\eta}^{\eps}(u)|^{p}\big]\le\mathsf L_p\,(1+\log(1/\eps))^{p/2}$ for all
	$\eps\in(0,1]$.
\end{lemma}

\begin{proof}
	With $\widehat{\boldsymbol{\eta}}(u)=\boldsymbol{\eta}^{\eps}(\eps^{2}u)$ (\Cref{sec:noise}),
	$\sup_{u\le t}|\boldsymbol{\eta}^{\eps}(u)|=\sup_{u\le t/\eps^{2}}|\widehat{\boldsymbol{\eta}}(u)|$. 
	Cover $[0,t/\eps^{2}]$ by the $m \triangleq \lceil t/\eps^{2}\rceil$ intervals $[j,j+1]$, 
	$0\le j\le m-1$.
	Since $\widehat{\boldsymbol{\eta}}$ is the stationary OU process with unit time scale,
	$\sup_{u\in[j,j+1]}|\widehat{\boldsymbol{\eta}}(u)|$ has the law of $Y_1$ for every $j$, 
	and a union bound together with \Cref{lem:OU-unit} gives \eqref{eq:OUmax}.
	
	For the moments, let $Y \triangleq \sup_{u\le T}|\boldsymbol{\eta}^{\eps}(u)|$, $m \triangleq \lceil T/\eps^{2}\rceil$ and
	$y_0 \triangleq (\log m/c_0)^{1/2}$. For $y\ge y_0$ we have $mC_0e^{-c_0y^{2}}\le
	C_0e^{-c_0(y^{2}-y_0^{2})}\le C_0e^{-c_0(y-y_0)^{2}}$, hence, for $p\ge1$,
	\[
	\E[Y^{p}]=\int_0^{\infty}py^{p-1}\PP(Y>y)\,\dd y
	\le y_0^{p}+C_0p\int_0^{\infty}(z+y_0)^{p-1}e^{-c_0z^{2}}\dd z
	\le\mathsf L'_p\,(1+y_0)^{p},
	\]
	using $(z+y_0)^{p-1}\le2^{(p-2)_+}(z^{p-1}+y_0^{p-1})$. For $\eps\in(0,1]$,
	$m\le(T+1)\eps^{-2}$, so that $y_0^{2}\le c_0^{-1}(\log(T+1)+2\log(1/\eps))
	\le C(1+\log(1/\eps))$, and the bound follows. For $p\in(0,1)$ use Jensen's inequality.
\end{proof}

\subsection{Two lemmas on observables}\label{app:obs}
Both lemmas are elementary. 
\Cref{lem:crossing} is the pathwise inequality underlying Avikainen-type estimates \cite{Avi09}, 
stated for the class $\BVJ$.
\begin{lemma}[Crossing inequality]\label{lem:crossing}
	Let $f\in\BVJ$ and let $\mu_f$ be as in \eqref{eq:muf}. Then, for all $x,y\in\R$,
	\[
	|f(x)-f(y)|\le|\mu_f|\big(\seg{x}{y}\big),\qquad
	\big(f(x)-f(y)\big)^{2}\le\TV(f)\,|\mu_f|\big(\seg{x}{y}\big).
	\]
\end{lemma}

\begin{proof}
	Let $J^{-} \triangleq \{x\in\R:f(x)\neq f_{+}(x)\}$, a countable set. 
	By condition (J), $f(x)=f_{-}(x)$ for $x\in J^{-}$, while $f(x)=f_+(x)$ for $x\notin J^{-}$. 
	Define $\chi(u,y) \triangleq \1_{\{u<y\}}$ if $u\in J^{-}$ 
	and $\chi(u,y) \triangleq \1_{\{u\le y\}}$ otherwise. 
	Then
	$\chi(u,y)=\1_{\{u\le y\}}-\1_{J^{-}}(y)\1_{\{u=y\}}$, so that, by \eqref{eq:muf},
	\[
	f(-\infty)+\int_\R\chi(u,y)\,\mu_f(\dd u)
	=f_+(y)-\1_{J^{-}}(y)\big(f_{+}(y)-f_{-}(y)\big)=f(y)\qquad(y\in\R).
	\]
	Hence $f(x)-f(y)=\int_\R(\chi(u,x)-\chi(u,y))\,\mu_f(\dd u)$. 
	For each $u$,
	$y\mapsto\chi(u,y)$ is the indicator of a half-line $(u,\infty)$ or $[u,\infty)$, 
	so $\chi(u,x)\ne\chi(u,y)$ implies $u\in\seg{x}{y}$. 
	Therefore $|f(x)-f(y)|\le\int_\R\1_{\seg{x}{y}}(u)\,|\mu_f|(\dd u)=|\mu_f|(\seg{x}{y})$, 
	and the second inequality follows from $|\mu_f|(\seg{x}{y})\le|\mu_f|(\R)\le\TV(f)$.
\end{proof}

\begin{lemma}[Cusp class]\label{lem:cusp}
	Let $f\in\Cusp$. Then:
	\begin{enumerate}[label= \rm{(\roman*)}]
		\item if $\seg{x}{y}\cap S=\emptyset$, then
		$|f(x)-f(y)|\le[f]_{\alpha,S}\,\dist\big(\seg{x}{y},S\big)^{\alpha-1}|x-y|$; 
		
		\item $f\in\Hol$ with $[f]_\alpha\le c_{\alpha,S}[f]_{\alpha,S}$, 
		where $c_{\alpha,S} \triangleq 2^{2-\alpha}|S|/\alpha$.
	\end{enumerate}
\end{lemma}

\begin{proof}
	Since $f$ is absolutely continuous on compact intervals,
	$|f(x)-f(y)|\le\int_{\seg{x}{y}}|f'(u)|\dd u\le[f]_{\alpha,S}\int_{\seg{x}{y}}\dist(u,S)^{\alpha-1}\dd u$.
	(i) follows from $\dist(u,S)^{\alpha-1}\le\dist(\seg{x}{y},S)^{\alpha-1}$ for $u\in\seg{x}{y}$. 
	For (ii), since $\alpha-1\le0$,
	$\dist(u,S)^{\alpha-1}=\max_{\varsigma\in S}|u-\varsigma|^{\alpha-1}
	\le\sum_{\varsigma\in S}|u-\varsigma|^{\alpha-1}$. 
	For an interval $J$ of length $h$ and
	$\varsigma\in\R$, splitting $J$ according to $|u-\varsigma|\le h/2$ or not,
	\[
	\int_J|u-\varsigma|^{\alpha-1}\dd u\le2\int_0^{h/2}z^{\alpha-1}\dd z+(h/2)^{\alpha-1}h
	=\frac{2^{1-\alpha}}{\alpha}h^{\alpha}+2^{1-\alpha}h^{\alpha}\le\frac{2^{2-\alpha}}{\alpha}h^{\alpha}.
	\]
	Applying this with $J=\seg{x}{y}$, $h=|x-y|$, and summing over $\varsigma\in S$ gives (ii).
\end{proof}

\section{The law of the white-noise velocity}\label{app:whitenoise}
To the best of our knowledge, \Cref{prop:comparison,cor:coupling} are new. 
\Cref{prop:density} gives an explicit bound on the density of $V^{0}(T)$; 
its relation to known results is discussed in \Cref{rem:density}.

\begin{proposition}[Pathwise comparison principle]\label{prop:comparison}
	Let $\gamma\ge0$ and $\ell\in\R$. For all $w,w'\in C^{0}_{0}([0,T])$ and $t\in[0,T]$,
	\begin{equation}\label{eq:comparison}
		\sup_{u\le t}\big|\mathcal S(w)(u)-\mathcal S(w')(u)\big|\le 2\sup_{u\le t}|w(u)-w'(u)| .
	\end{equation}
	In particular, $\mathcal S$ is $2$-Lipschitz on $(C^{0}_{0}([0,T]),\|\cdot\|_{\infty})$.
\end{proposition}
\begin{proof}
	Let $(v,k)$ and $(v',k')$ be the solutions of \eqref{eq:SP} for $w$ and $w'$, and put
	$g \triangleq k-\bb(v)$, $g' \triangleq k'-\bb(v')$, which are bounded measurable functions on $[0,T]$, and
	$h \triangleq w-w'$. By \eqref{eq:SP-integral}, the function
	$\Lambda \triangleq (v-w)-(v'-w')=-\int_0^{\cdot}(g-g')\dd u$ is absolutely continuous,
	$\Lambda(0)=0$, $\Lambda'=-(g-g')$ a.e., and $v-v'=\Lambda+h$. For a.e.\ $u$,
	\[
	\big(g(u)-g'(u)\big)\big(v(u)-v'(u)\big)=\big(k(u)-k'(u)\big)\big(v(u)-v'(u)\big)
	+\gamma\big(v(u)-v'(u)\big)^{2}\ge0,
	\]
	by the monotonicity of $\partial\varphi$; in particular $g(u)\ge g'(u)$ whenever
	$v(u)>v'(u)$. Fix $t\in[0,T]$ and let $H \triangleq \sup_{u\le t}|h(u)|$. Suppose that
	$\Lambda(t_1)>H$ for some $t_1\le t$, and let $\tau \triangleq \sup\{u\in[0,t_1]:\Lambda(u)\le H\}$;
	the set is closed and contains $0$, so $\Lambda(\tau)\le H$, $\tau<t_1$ and $\Lambda>H$ on
	$(\tau,t_1]$. On $(\tau,t_1]$ we have $v-v'=\Lambda+h>H-|h|\ge0$, hence $g\ge g'$ and
	$\Lambda'\le0$ a.e., so $\Lambda(t_1)\le\Lambda(\tau)\le H$, a contradiction. Thus
	$\Lambda\le H$ on $[0,t]$, and by symmetry $\Lambda\ge-H$. Consequently
	$|v(u)-v'(u)|\le|\Lambda(u)|+|h(u)|\le2H$ for $u\le t$.
\end{proof}

\begin{proposition}[Coupling error]\label{cor:coupling}
	Assume that \rm{Assumptions \ref{ass:vel}} and \rm{\ref{ass:colo}} (but not necessarily \rm{\Cref{ass:nondeg}}) hold , 
	and let $C_0,c_0>0$ be the constants of \rm{\Cref{lem:OUmax}}, 
	which depend only on $\mathbf{A},\mathbf{K},d,d'$. 
	Then:
	\begin{enumerate}[label= \rm{(\roman*)},itemsep=1pt,topsep=2pt]
		\item almost surely,
		$\sup_{t\le T}|Z^{\eps}(t)|\le2\sup_{t\le T}|h^{\eps}(t)|\le4\kappa_0\,\eps\sup_{t\le T}|\boldsymbol{\eta}^{\eps}(t)|$;
		\item for all $x\ge0$,
		$\PP\big(\sup_{t\le T}|Z^{\eps}(t)|>x\big)\le\lceil T/\eps^{2}\rceil\,C_0\exp\big(-c_0x^{2}/(16\kappa_0^{2}\eps^{2})\big)$;
		\item for every $p>0$ there is $\mathsf C_p$ such that
		$\E\big[\sup_{t\le T}|Z^{\eps}(t)|^{p}\big]\le\mathsf C_p\,\vth^{\,p}$ for all $\eps\in(0,1]$;
		\item for every $q>0$, with $\mathsf K_q \triangleq 4\kappa_0\,((q+2)/c_0)^{1/2}$,
		\begin{equation*}\label{eq:tail-q}
			\PP\Big(\sup_{t\le T}|Z^{\eps}(t)|>\mathsf K_q\vth\Big)\le(T+1)\,C_0\,\eps^{q},\qquad\eps\in(0,1].
		\end{equation*}
	\end{enumerate}
\end{proposition}

\begin{proof}
	(i) By \eqref{eq:coupled}, $V^{\eps}=\mathcal S(N^{\eps})$ and $V^{0}=\mathcal S(N^{0})$ with
	$N^{\eps},N^{0}\in C^{0}_{0}([0,T])$, so Proposition~\ref{prop:comparison} gives the first
	inequality, and \eqref{eq:h-bound} the second. (ii) follows from (i) and \eqref{eq:OUmax},
	and (iii) from (i) and the moment bound of Lemma~\ref{lem:OUmax}. (iv) For
	$x=\mathsf K_q\vth$ we get $c_0x^{2}/(16\kappa_0^{2}\eps^{2})=(q+2)(1+\log(1/\eps))$, and
	$\lceil T/\eps^{2}\rceil\le(T+1)\eps^{-2}$ for $\eps\le1$, so (ii) gives
	$\PP(\cdot)\le(T+1)\eps^{-2}C_0e^{-(q+2)}\eps^{q+2}\le(T+1)C_0\eps^{q}$.

	\Cref{cor:coupling} strengthens, with a rate, the convergence in probability of
	\cite[Proposition~5.1(2)]{GM22}; the factor $(1+\log(1/\eps))^{1/2}$, due to the supremum of
	the OU process over a time interval of length $T/\eps^{2}$, is of the correct order for
	scalar noise \cite{GP00}.
\end{proof}

\begin{proposition}[Density of the white-noise velocity]\label{prop:density}
	Assume that \rm{Assumptions \ref{ass:vel}, \ref{ass:colo}}, and \rm{\ref{ass:nondeg}} hold, 
	and let $U,U_*$ be as in \eqref{eq:potential}. Then the law of
	$V^{0}(T)$ is equivalent to Lebesgue measure, and its density $p_T$ satisfies $p_T\le\bar p_T$
	a.e., where $\bar p_T$ is the explicit function 
	\begin{equation}\label{eq:pbar}
		\bar p_T(v) \triangleq e^{\frac{\gamma T}{2}+\frac{U(v_0)-U(v)}{\mathfrak{s}^{2}}}
		\Big[\frac{e^{-\frac{(v-v_0)^{2}}{2\mathfrak{s}^{2}T}}}{\sqrt{2\pi \mathfrak{s}^{2}T}}
		+\frac{\Delta}{\mathfrak{s}^{2}}\,e^{\frac{\Delta^{2}T}{2\mathfrak{s}^{2}}-\frac{\Delta(|v|+|v_0|)}{\mathfrak{s}^{2}}}\,
		\Phibar\Big(\frac{|v|+|v_0|-\Delta T}{\mathfrak{s}\sqrt T}\Big)\Big], \quad
		v\in \R.
	\end{equation}
	In particular,
	\begin{equation}\label{eq:M}
		M \triangleq \operatorname*{ess\,sup}_{v\in\R}p_T(v)\le
		e^{\frac{\gamma T}{2}+\frac{U(v_0)-U_*}{\mathfrak{s}^{2}}}
		\Big[\frac{1}{\sqrt{2\pi \mathfrak{s}^{2}T}}+\frac{\Delta}{\mathfrak{s}^{2}}\,e^{\frac{\Delta^{2}T}{2\mathfrak{s}^{2}}}\Big]<\infty .
	\end{equation}
\end{proposition}

\begin{proof}
	Let $\drift$ be as in Lemma~\ref{lem:nostick}, so that \eqref{eq:white-sde} holds, and let
	$U_0(v)\triangleq\frac{\gamma}{2}v^{2}-\ell v$, so that $U=U_0+\Delta|\cdot|$ and $\bb=-U_0'$.
	We show that, for every Borel set $E\subset\R$,
	\begin{equation}\label{eq:density-bound}
		\PP\big(V^{0}(T)\in E\big)\le\int_E\bar p_T(v)\,\dd v,
	\end{equation}
	and deduce the proposition from \eqref{eq:density-bound} in Step~4.
	
	\emph{Step 1 (Girsanov transformation to pure dry friction).} Let
	$\Theta_u\triangleq\bb(V^{0}(u))/\mathfrak{s}$. Since $|\bb(v)|\le\gamma|v|+|\ell|$ and, by
	\eqref{eq:comparison} applied on $[0,u]$,
	$|V^{0}(u)|\le\|\mathcal S(0)\|_{\infty}+2\mathfrak{s}\sup_{r\le u}|B(r)|$, there is $C_\Theta$
	such that $|\Theta_u|\le C_\Theta(1+\sup_{r\le u}|B(r)|)$. Since $V^{0}$ is a non-anticipative
	functional of $B$, Bene\v{s}' condition \cite[Chapter~3, Corollary~5.16]{KS91} shows that
	\[
	\mathcal Z_t\triangleq\exp\Big(-\int_0^{t}\Theta_u\,\dd B(u)-\frac12\int_0^{t}\Theta_u^{2}\,\dd u\Big),
	\qquad t\in[0,T],
	\]
	is a martingale. Define $\QQ\triangleq\mathcal Z_T\cdot\PP$ on $\mathcal F_T$. By Girsanov's
	theorem \cite[Chapter~3, Theorem~5.1]{KS91}, $\widetilde B(t)\triangleq B(t)+\int_0^{t}\Theta_u\,\dd u$
	is a $\QQ$-Brownian motion on $[0,T]$, and \eqref{eq:white-sde} becomes
	\[
	V^{0}(t)=v_0-\Delta\int_0^{t}\sgn\big(V^{0}(u)\big)\,\dd u+\mathfrak{s}\widetilde B(t),\qquad t\in[0,T].
	\]
	Thus, under $\QQ$, $V^{0}$ is a weak solution of the white-noise equation of pure dry friction
	(regime~(I)). After division by $\mathfrak{s}$, this equation has unit diffusion coefficient and
	a bounded drift, so that its weak solutions are unique in law
	\cite[Chapter~5, Proposition~3.10]{KS91}; 
	the value of $\sgn(0)$ is irrelevant by Lemma~\ref{lem:nostick}. 
	Hence, under $\QQ$, $V^{0}(T)$ has the closed-form density of pure dry friction \cite[(2.10)]{TVJ10},
	\begin{equation}\label{eq:pI}
		p^{\rm (I)}_T(v)\triangleq e^{-\frac{\Delta^{2}T}{2\mathfrak{s}^{2}}-\frac{\Delta(|v|-|v_0|)}{\mathfrak{s}^{2}}}\,
		\frac{e^{-\frac{(v-v_0)^{2}}{2\mathfrak{s}^{2}T}}}{\sqrt{2\pi\mathfrak{s}^{2}T}}
		+\frac{\Delta}{\mathfrak{s}^{2}}\,e^{-\frac{2\Delta|v|}{\mathfrak{s}^{2}}}\,
		\Phibar\Big(\frac{|v|+|v_0|-\Delta T}{\mathfrak{s}\sqrt T}\Big),\qquad v\in\R.
	\end{equation}
	Since $\mathcal Z_T>0$, $\PP\ll\QQ$, and since
	$\mathfrak{s}\,\dd\widetilde B(u)=\dd V^{0}(u)+\Delta\sgn(V^{0}(u))\,\dd u$,
	\begin{equation*}\label{eq:dPdQ}
		\begin{aligned}
			\frac{\dd\PP}{\dd\QQ}=\mathcal Z_T^{-1}
			&=\exp\Big(\int_0^{T}\Theta_u\,\dd\widetilde B(u)-\frac12\int_0^{T}\Theta_u^{2}\,\dd u\Big)\\
			&=\exp\Big(\frac{1}{\mathfrak{s}^{2}}\int_0^{T}\bb(V^{0})\,\dd V^{0}
			+\frac{1}{\mathfrak{s}^{2}}\int_0^{T}\Big(\Delta\,\bb(V^{0})\sgn(V^{0})-\frac12\,\bb(V^{0})^{2}\Big)\dd u\Big).
		\end{aligned}
	\end{equation*}
	
	\emph{Step 2 (It\^o's formula).} Since $U_0''=\gamma$ and $\langle V^{0}\rangle_t=\mathfrak{s}^{2}t$,
	It\^o's formula gives
	$\int_0^{T}\bb(V^{0})\,\dd V^{0}=U_0(v_0)-U_0(V^{0}(T))+\gamma\mathfrak{s}^{2}T/2$. Moreover,
	$\Delta\,\bb\sgn-\bb^{2}/2\le\Delta|\bb|-\bb^{2}/2\le\Delta^{2}/2$. Hence
	\begin{equation}\label{eq:dPdQ-bound}
		\frac{\dd\PP}{\dd\QQ}\le\exp\Big(\frac{\gamma T}{2}+\frac{\Delta^{2}T}{2\mathfrak{s}^{2}}
		+\frac{U_0(v_0)-U_0(V^{0}(T))}{\mathfrak{s}^{2}}\Big).
	\end{equation}
	
	\emph{Step 3 (law under $\QQ$).} By \eqref{eq:dPdQ-bound} and Step~1,
	\[
	\PP\big(V^{0}(T)\in E\big)=\E_{\QQ}\Big[\frac{\dd\PP}{\dd\QQ}\,\1_E(V^{0}(T))\Big]
	\le\int_E e^{\frac{\gamma T}{2}+\frac{\Delta^{2}T}{2\mathfrak{s}^{2}}+\frac{U_0(v_0)-U_0(v)}{\mathfrak{s}^{2}}}\,
	p^{\rm (I)}_T(v)\,\dd v=\int_E\bar p_T(v)\,\dd v,
	\]
	where the last equality follows from \eqref{eq:pI}, \eqref{eq:pbar} and $U=U_0+\Delta|\cdot|$.
	This proves \eqref{eq:density-bound}.
	
	\emph{Step 4 (consequences).} By \eqref{eq:density-bound}, $\PP(V^{0}(T)\in E)=0$ whenever
	$\Leb(E)=0$, so $V^{0}(T)$ has a density $p_T$, and $p_T\le\bar p_T$ a.e.\ (apply
	\eqref{eq:density-bound} to $E=\{p_T>\bar p_T\}$). The bound \eqref{eq:M} follows from
	$U(v)\ge U_*$, $e^{-\Delta(|v|+|v_0|)/\mathfrak{s}^{2}}\le1$, $\Phibar\le1$ and the bound $1$ for the
	Gaussian factor. Finally, if $\Leb(E)>0$, then $\QQ(V^{0}(T)\in E)=\int_E p^{\rm (I)}_T(v)\,\dd v>0$
	and $\dd\PP/\dd\QQ>0$, so that $\PP(V^{0}(T)\in E)>0$. Hence the law of $V^{0}(T)$ is equivalent
	to Lebesgue measure.
	
\end{proof}
\begin{remark}\label{rem:density}
	(i) By \eqref{eq:pI} and \eqref{eq:pbar},
	\[
	\bar p_T(v)=e^{\frac{\gamma T}{2}+\frac{\Delta^{2}T}{2\mathfrak s^{2}}+\frac{U_0(v_0)-U_0(v)}{\mathfrak s^{2}}}\,
	p^{\rm (I)}_T(v),\qquad v\in\R.
	\]
	In regime~(I), $\bb=0$, so that $\QQ=\PP$ and
	$p_T=p^{\rm (I)}_T=e^{-\Delta^{2}T/(2\mathfrak s^{2})}\,\bar p_T$.
	
	(ii) In regimes~(II)--(III), the propagator is known through spectral series \cite[(3.12), (4.7)]{TVJ10}, 
	whose completeness is not proved there, and through its Laplace transform in time, 
	given in closed form in \cite[(7)--(9)]{TPJ12}. 
	Neither representation yields an explicit bound on $p_T$.
	
	(iii) The boundedness and the strict positivity of $p_T$ also follow from the two-sided
	Gaussian estimates of \cite{MPZ21} for stochastic differential equations with measurable drift
	of linear growth, applied to \eqref{eq:white-sde}; the constants there are not explicit.
	
	(iv) The closed form \eqref{eq:pI} was derived in \cite{KS84} by a Girsanov transformation to
	Brownian motion and the joint law of Brownian motion, its local time and its occupation time;
	the same Girsanov--local-time representation underlies the exact simulation algorithm of
	\cite[Section~3]{EM14}. 
	Comparing $V^{0}$ with pure dry friction instead of Brownian motion, 
	as in the proof above, keeps the local time inside \eqref{eq:pI}, 
	so that only It\^o's formula for the smooth part $U_0$ of $U$ is needed.
\end{remark}

\section{Implementation details}\label{app:implementation}
\paragraph{Algorithm.} Algorithm~\ref{alg:coupled} summarizes the simulation of one coupled
path with the backward Euler scheme \eqref{eq:BE}.

\begin{algorithm}[ht]
	\caption{Coupled simulation of $(V^{\eps}(T),V^{0}(T))$ (backward Euler variant)}\label{alg:coupled}
	\begin{algorithmic}[1]
		\Require $\eps$, $T$, $n$ (so $\delta t=T/n$), $\mathbf{A},\mathbf{K},\mathbf{r},\sigma$, 
		$\Delta,\gamma,\ell$, $v_0$
		\State draw $\boldsymbol{\eta}\sim\mathcal N(0,\mathbf C)$;
		set $V^{\eps}\gets v_0$, $V^{0}\gets v_0$, {$\mathbf{c}$}
		$\gets \mathbf{A}^{-\top}\mathbf{r}$, $h\gets0$, $Z^{*}\gets0$, $H^{*}\gets0$
		\For{$m=0,\dots,n-1$}
		\State draw $\boldsymbol{\zeta}\sim\mathcal N(0,I_{d'})$ 
		\Comment{one Gaussian vector drives both processes}
		\State $\delta N^{\eps}\gets\sigma\,\delta t\,\mathbf{r}\cdot\boldsymbol{\eta}/\eps$; \quad
		$\delta N^{0}\gets\sigma\sqrt{\delta t}\,\mathbf{c}^{\top}\mathbf{K}\boldsymbol{\zeta}$; \quad 
		$h\gets h+\delta N^{\eps}-\delta N^{0}$
		\State $V^{\eps}\gets\operatorname{soft}_{\Delta\delta t}(V^{\eps}+\ell\delta t+\delta N^{\eps})/(1+\gamma\delta t)$
		\State $V^{0}\gets\operatorname{soft}_{\Delta\delta t}(V^{0}+\ell\delta t+\delta N^{0})/(1+\gamma\delta t)$
		\State $\boldsymbol{\eta}\gets\boldsymbol{\eta}-(\delta t/\eps^{2})\mathbf{A}\boldsymbol{\eta}+(\sqrt{\delta t}/\eps)\mathbf{K}\boldsymbol{\zeta}$ \Comment{same $\boldsymbol{\zeta}$: this is the coupling}
		\State {$Z^{*}\gets\max\{Z^{*},|V^{\eps}-V^{0}|\}$; \quad $H^{*}\gets\max\{H^{*},|h|\}$} 
		\Comment{{$Z^{*}\le2H^{*}$ by \Cref{prop:discrete}}}
		\EndFor
		\State \Return $(V^{\eps},V^{0}{,Z^{*},H^{*}})$
	\end{algorithmic}
\end{algorithm}

\paragraph{Semi-implicit variant.} 
The update
$V_{m+1}=\operatorname{soft}_{\Delta\delta t}\big((1-\gamma\delta t)V_m+\ell\,\delta t+\delta N_m\big)$,
that is, $V_{m+1}=V_m+\delta N_m+\delta t\,(\bb(V_m)-k_{m+1})$ with
$k_{m+1}\in\partial\varphi(V_{m+1})$, treats the viscous friction explicitly. 
If $\gamma\delta t\le1$, \Cref{prop:discrete} holds for it as well: 
with the notation of its proof,
$\Lambda_{m+1}=(1-\gamma\delta t)\Lambda_m-\gamma\delta t\,h_m-\delta t\,(k^{\eps}_{m+1}-k^{0}_{m+1})$,
and $\Lambda_{m+1}>H_{m+1}$ would imply $k^{\eps}_{m+1}\ge k^{0}_{m+1}$, 
hence $\Lambda_{m+1}\le(1-\gamma\delta t)H_m+\gamma\delta t\,H_m=H_m$, a contradiction. 
The results of Section~\ref{sec:results} are obtained with \eqref{eq:BE}.
	
\paragraph{Time step.} 
We take $\delta t=T/n$ with $n=\lceil T/\min\{10^{-3},\eps^{2}/50\}\rceil$,
so that, for the test problem ($\mathbf{A}=1$), $\delta t\,\mathbf{A}/\eps^{2}\le1/50$ 
for every $\eps$ (for a general matrix $\mathbf{A}$ one takes $\delta t\le\eps^{2}/(50\|\mathbf{A}\|)$):
the OU recursion is stable, it resolves the correlation time in the same way at every $\eps$ for which $\eps^{2}/50\le10^{-3}$ 
(the contraction factor $1-\delta t\mathbf{A}/\eps^{2}$ is then $49/50$ up to the rounding of $n$), 
and in the scalar case its stationary variance exceeds $\mathbf{K}^{2}/(2\mathbf{A})$ 
by the factor $(1-\delta t\mathbf{A}/(2\eps^{2}))^{-1}\le 1.0102$, 
which is the same for all these $\eps$ and therefore does not affect the fitted exponents.

\paragraph{Random numbers and standard errors.} 
Batch $\mathfrak{B}$ of the cell (stage, regime,~$\eps$) uses the PCG64 generator \cite{ONe14} of NumPy, 
seeded by \texttt{SeedSequence}$(42424242$, \texttt{spawn\_key}$=($stage, regime, $\mathrm{round}(10^{6}\eps)$, $\mathfrak{B}))$, 
so that all streams are independent and the results are reproducible. 
With $s(\cdot)$ the sample standard deviation over the $N$ paths of a cell, 
the standard errors are $s((D^{\eps})^{2})N^{-1/2}$ for $\mathcal D_2$; 
$s(F^{\eps})N^{-1/2}$, $s(D^{\eps})N^{-1/2}$ and $\widehat\sigma_{K}N^{-1/2}$ for $\widehat{\mathrm{I}}^{\eps}_N$, 
$\widehat J^{\eps}_N$ and $\widehat K^{\eps}_N$, 
where $\widehat\sigma_{K}^{2}=s(F^{\eps})^{2}-\widehat\lambda^{\eps}_N\,\widehat{\Cov}(F^{\eps},F^{0})$; 
and $s(D^{\eps})N^{-1/2}$ for the bias $\E [D^{\eps}]$. 
Variances, $\widehat\lambda^{\eps}_N$ and the variance ratios are computed from all $N$ paths; 
their standard errors are $s_{\mathfrak{B}}/\sqrt{\mathfrak{B}}$, 
with $s_{\mathfrak{B}}$ the standard deviation of the $\mathfrak{B}=12$ batch estimates.

\paragraph{Local slopes and fitted exponents.} 
For consecutive grid values $\eps_i>\eps_{i+1}$, the local slope
\[
r_i \triangleq \frac{\log\mathcal D_2(\eps_{i+1})-\log\mathcal D_2(\eps_i)}{\log\eps_{i+1}-\log\eps_i}
\]
is plotted at $\bar\eps_i \triangleq (\eps_i\eps_{i+1})^{1/2}$. 
Since the cells are independent, its
standard error is approximately $(e_i^{2}+e_{i+1}^{2})^{1/2}/|\log(\eps_{i+1}/\eps_i)|$, 
where $e_i$ is the relative standard error of $\mathcal D_2(\eps_i)$. 
The fitted exponent is the weighted least-squares slope of $\log\mathcal D_2(\eps_i)$ 
against $\log\eps_i$ over the $k$ grid values $\eps_i\le 0.05$, with weights $e_i^{-2}$. 
Its standard error is $\big(\sum_i e_i^{-2}(\log\eps_i-\overline{\log\eps})^{2}\big)^{-1/2}$, 
where $\overline{\log\eps}$ is the weighted mean of the $\log\eps_i$, 
multiplied by $(\chi^{2}/(k-2))^{1/2}$ when this factor exceeds $1$, 
$\chi^{2}$ being the weighted residual sum of squares. 
The exponents of the coupling errors and of the bias are fitted in the same way.

\paragraph{Reference values.} 
By It\^o's formula and Lemma~\ref{lem:nostick}, 
the laws of $V^{0}(t)$, $t\in[0,T]$, form a weak solution of the Fokker--Planck equation
\begin{equation}\label{eq:FP}
	\partial_t p=\partial_v\big[(\gamma v+\Delta\sgn(v)-\ell)\,p\big]+\frac{\mathfrak{s}^{2}}{2}\,\partial_v^{2}p,
	\qquad p(0,\cdot)=\delta_{v_0},
\end{equation}
which is \cite[(1.4)]{TVJ10} with noise power $\mathfrak{s}^{2}$ and external force $\ell$.
In regime (I), $p_T$ is given by \cite[(2.10)]{TVJ10}, that is, by \eqref{eq:pI}.
In regimes (II)--(III), it is evaluated from the spectral series \cite[(3.12), (4.7)]{TVJ10},
truncated at $40$ modes and converted to physical variables. 

The eigenvalues are bracketed on a grid of step $0.05$ and refined by Brent's method. 
The parabolic cylinder functions $D_\nu$ are evaluated with SciPy \cite{Vir20}; 
since the library routine loses accuracy for orders slightly above a positive integer, 
$D_\nu$ is obtained by five-point Lagrange interpolation in $\nu$ when $\nu$ lies within $10^{-4}$ of an integer.

Since the completeness of these eigenfunction expansions is not proved in \cite{TVJ10}, 
the truncated series are cross-validated against three independent computations: 
(i) the numerical inversion, by Talbot's method \cite{Tal79} as in \cite[Section~2]{TPJ12}, 
of the closed-form Laplace transform \cite[(7)--(9)]{TPJ12}, 
evaluated with \texttt{mpmath} in $30$-digit arithmetic at the test points $v\in\{-2,-0.5,0.25,1,2.5\}$;  
(ii) a finite-volume solution of \eqref{eq:FP} on $[-20,20]$ with mesh size $0.0025$ 
and a cell centred at $0$, Scharfetter--Gummel fluxes \cite{SG69}, zero-flux boundary conditions, 
and implicit Euler steps $4\cdot10^{-4}$ and $2\cdot10^{-4}$ combined by Richardson extrapolation; 
(iii) and the sample means of $f(V^{0}_n)$ in all cells of the experiment. 
Quadratures are split at the non-smooth points of the integrand ($0$, $\pm v_f$ and $\pm(v_f+1)$ for the severity family); 
for the finite-volume values, the density is interpolated by cubic splines on each side of $0$. 
Table~\ref{tab:ref} reports the resulting reference values, the discrepancies, and $\sup p_T$,
which the bound \eqref{eq:M} exceeds by a factor between $2.0$ and $2.7$; 
the inequality $p_T\le\bar p_T$ of \Cref{prop:density} was also checked on a grid of $[-10,10]$.

\begin{table}[H]
\centering
\caption{Reference values $\mathrm{I}^{0}=\E[f(V^{0}(T))]$ for the seven test observables in regimes (I)--(III), 
and checks of the density $p_T$: 
the change of $\mathrm{I}^{0}$ (maximum over the fifteen observables) from $30$ to $40$ modes; 
the largest discrepancies of $p_T$ with \cite[(2.10)]{TVJ10} on a grid of $[-6,6]$, 
and with the inverted Laplace transform of \cite{TPJ12} and the finite-volume solution at the five test points; 
the largest discrepancy between $\mathrm{I}^{0}$ and its finite-volume value over the fifteen observables; 
the normalization; and $\sup p_T$ against the bound \eqref{eq:M}.}\label{tab:ref}
\begin{tabular}{@{}lccc@{}}
	\toprule
	& regime (I) & regime (II) & regime (III)\\
	\midrule
	\multicolumn{4}{@{}l}{$\mathrm I^{0}=\E[f(V^{0}(T))]$}\\
	$f_{\rm ind}$ & $0.05809535$ & $0.00769924$ & $0.01339634$\\
	$f_{1/4}$ & $0.04442633$ & $0.00507531$ & $0.00908520$\\
	$f_{1/3}$ & $0.04121140$ & $0.00449086$ & $0.00811278$\\
	$f_{1/2}$ & $0.03606222$ & $0.00358813$ & $0.00659817$\\
	$f_{3/4}$ & $0.03053359$ & $0.00267426$ & $0.00504382$\\
	$f_{1}$ & $0.02666738$ & $0.00207778$ & $0.00401286$\\
	$f_{\rm sm}$ & $0.19463284$ & $0.10696800$ & $0.12412701$\\
	\midrule
	\multicolumn{4}{@{}l}{checks of the density $p_T$}\\
	modes of the spectral series & -- & $40$ & $40$\\
	change of $\mathrm I^{0}$ from $30$ to $40$ modes & -- & $8.7\cdot10^{-19}$ & $1.4\cdot10^{-17}$\\
	\eqref{eq:pI} against \cite[(2.10)]{TVJ10} & $1.1\cdot10^{-16}$ & -- & --\\
	Laplace inversion \cite{TPJ12} & -- & $6.1\cdot10^{-16}$ & $2.2\cdot10^{-16}$\\
	finite volumes, $\max|p_T-p_T^{\rm FV}|$ & $2.3\cdot10^{-7}$ & $3.9\cdot10^{-7}$ & $4.2\cdot10^{-7}$\\
	finite volumes, $\max_f|\mathrm I^{0}-\mathrm I^{0}_{\rm FV}|$ & $1.6\cdot10^{-7}$ & $9.8\cdot10^{-8}$ & $9.8\cdot10^{-8}$\\
	$\int_\R p_T\,\dd v-1$ & $0$ & $-3.3\cdot10^{-16}$ & $-1.1\cdot10^{-16}$\\
	$\sup_v p_T(v)$ & $0.584$ & $0.768$ & $0.718$\\
	bound \eqref{eq:M} on $M$ & $1.187$ & $2.083$ & $1.838$\\
	\bottomrule
\end{tabular}
\end{table}
\end{document}